\documentclass[11pt]{article}
\usepackage{amsmath,amsthm,amssymb,amscd}
\usepackage{xurl}
\usepackage{hyperref}
\usepackage{geometry}
\usepackage{tikz-cd}
\usepackage{booktabs} % For better table aesthetics
\usepackage{graphicx} % For including images
\usepackage{multirow}
\usepackage{tikz}
\usetikzlibrary{calc,positioning,3d}
\usepackage[normalem]{ulem} % cross through
\usepackage{float}

\usepackage{subcaption}
\usepackage{caption}

\usepackage{makecell} % Add this in the preamble

\newtheorem{theorem}{Theorem}[section]
\newtheorem{lemma}[theorem]{Lemma}
\newtheorem{proposition}[theorem]{Proposition}
\newtheorem{corollary}[theorem]{Corollary}
\theoremstyle{definition}
\newtheorem{definition}[theorem]{Definition}
\newtheorem{example}[theorem]{Example}

\title{Persistent commutative algebra on graphs and hypergraphs}
\author{
	Faisal Suwayyid\footnote{suwayyid@msu.edu~}$^{~1}$, 
		and Guo-Wei Wei\footnote{weig@msu.edu~}$^{~2,3,4}$ \\		
		$^1$Department of Mathematics,\\
		King Fahd University of Petroleum and Minerals, Dhahran 31261, KSA.\\
			$^2$Department of Mathematics,\\
		Michigan State University, MI 48824, USA.\\
		$^3$Department of Electrical and Computer Engineering,\\
		Michigan State University, MI 48824, USA.\\
		$^4$Department of Biochemistry and Molecular Biology,\\
		Michigan State University, MI 48824, USA.
	}
	
	\date{\today}
\begin{document}
 \maketitle 

\begin{abstract}

We introduce a persistent commutative algebra for studying the algebraic and combinatorial evolution of edge ideals of graphs and hypergraphs under filtration. Building on the Persistent Stanley--Reisner Theory (PSRT), we develop the notion of {persistent edge ideals} and analyze their graded Betti numbers across the filtration of graphs or hypergraphs. To enable this analysis, we establish a persistent extension of Hochster’s formula, providing a functorial correspondence between algebraic and topological persistence. We further examine the behavior of Betti splittings in the persistent setting, proving a general inequality that extends the classical splitting result to the filtration of monomial ideals. Motivated by graph-theoretic interpretations, we introduce {persistent minimal vertex covers}, which encode the temporal structure of combinatorial dependencies within evolving graphs or hypergraphs. Applications to alignment-free genomic classification and molecular isomer discrimination demonstrate the interpretability and representatbility of persistent edge ideals as algebraic invariants, bridging combinatorial commutative algebra and data science.

\end{abstract}

{\setcounter{tocdepth}{4} \tableofcontents}
\setcounter{page}{1}
\newpage	

\clearpage
\section{Introduction}
\label{sec:introduction}

Topological Data Analysis (TDA) provides a rigorous framework for extracting topological structures from high-dimensional and complex data. Among its most powerful tools is {persistent (co)homology}, which tracks the evolution of topological features---connected components, cycles, and higher-dimensional voids---across a family of simplicial complexes over filtration \cite{edelsbrunner2008persistent, BubenikDlotko2017PersistenceLandscapes, zomorodian2005computing}. Filtration-based constructions such as the Vietoris--Rips and Čech complexes encode data across multiple scales, producing persistence modules whose barcodes and diagrams quantify the lifespan of these features \cite{carlsson2009topology, AdamsCoskunuzer2022_GeometricApproachesPersistentHomology}. Persistent homology and its variants have been successfully applied to diverse problems, from protein folding and drug discovery to signal processing and materials science \cite{CangWei2017_TopologyNet, chen2023persistent, GrbicWuXiaWei2022_AspectsTopologicalApproaches,wee2025review}. Despite their success, homological approaches capture only {topological} variation and cannot detect some structural or non-topological differences---for instance, distinguishing a five-member ring from a six-member ring in data \cite{wei2025persistent,su2025topological}. 

A parallel algebraic development has emerged from {Combinatorial Commutative Algebra (CCA)}, which connects the combinatorics to the algebra of monomial ideals. The {Stanley--Reisner correspondence} encodes a simplicial complex as a square-free monomial ideal, while Hochster’s formula expresses its graded Betti numbers in terms of reduced cohomology of induced subcomplexes \cite{MillerSturmfels2005, bruns1998cohen, Hu2025CommAlgTDA}. This correspondence bridges algebraic geometry, topology, and combinatorics, providing refined algebraic invariants that complement topological invariants.

Building on this foundation, {Persistent Stanley--Reisner Theory (PSRT)}~\cite{SuwayyidWei2026_PSRT} unified persistent homology and combinatorial commutative algebra by studying the evolution of Stanley--Reisner ideals under filtration. PSRT introduced persistent graded Betti numbers, persistent $f$- and $h$-vectors, and facet persistence barcodes, demonstrating stability theorems analogous to those in classical persistent homology. These invariants captured multiscale algebraic changes in simplicial complexes, providing an algebraically grounded alternative to topology-based persistence.

The PSRT framework stimulates a broader program of {persistent commutative algebra}, which underlies several recent advances in commutative algebra learning. Notably, the {Commutative Algebra $k$-mer Learning (CAKL)} framework~\cite{SuwayyidHozumiFengZiaWeeWei2025_CAKL} applies persistent algebraic invariants to genomic classification and viral phylogeny, while {Commutative Algebra Machine Learning (CAML)}~\cite{FengSuwayyidZiaWeeHozumiChenWei2025_CAML} extends this paradigm to protein--ligand binding affinity prediction. Commutative algebra offers successful predictions of protein-nucleic
acid binding affinities \cite{zia2025cap}.
 More recently, the {Commutative Algebra Neural Network (CANet)}~\cite{WeeSuwayyidZiaFengHozumiWei2025_CANet} revealed genetic origins of diseases through persistent algebraic descriptors derived from DNA data. The use of graded Betti curves was explored in computational biology \cite{zia2025gbnl}. A comparative study of persistent commutative algebra, persistent homology, and topological spectral theory was reported \cite{ren2025interpretability}. 
Together, these works demonstrate that persistence in commutative algebra provides a unified, interpretable, and scalable mathematical foundation for multiscale learning across molecular and biological domains.

In the present work, we focus on the algebraic and combinatorial evolution of {edge ideals} under filtration, which introduces the notion of {persistent edge ideals of graphs and hypergraphs}. To establish their theoretical foundations, we extend Hochster’s formula to the persistent setting, linking algebraic persistence to topological persistence through induced maps in cohomology of subcomplexes. We further investigate the behavior of Betti splittings across filtration, proving a general inequality that extends the classical decomposition principle to the persistent setting. In addition, we define {persistent minimal vertex covers}, which encode the temporal dynamics of combinatorial dependencies within evolving graphs. These developments yield a persistent Betti theory for edge ideals, connecting algebraic, topological, and graph-theoretic aspects of persistence and demonstrating their interpretability and representability in applications to genomic classification and molecular isomer discrimination.

\section{Preliminaries}
\subsection{Graded and Multigraded Betti Numbers}

Let $S=k[x_1,\dots,x_n]$ be a polynomial ring over a field $k$. Throughout, we use the standard $\mathbb{Z}$–grading on $S$, $\deg(x_i)=1$, and consider finitely generated $\mathbb{Z}$–graded $S$–modules.

\smallskip

%\noindent\textbf{Minimal graded free resolutions and graded Betti numbers.}
Given a finitely generated $\mathbb{Z}$–graded $S$–module $M$, a minimal graded free resolution packages the homological information of $M$ in a way that respects degrees:
\[
0 \longleftarrow M \longleftarrow 
\bigoplus_{j\in\mathbb{Z}} S(-j)^{\beta_{0,j}(M)}
\longleftarrow 
\bigoplus_{j\in\mathbb{Z}} S(-j)^{\beta_{1,j}(M)}
\longleftarrow \cdots
\longleftarrow 
\bigoplus_{j\in\mathbb{Z}} S(-j)^{\beta_{i,j}(M)}
\longleftarrow \cdots \longleftarrow 0 .
\]
The integers $\beta_{i,j}(M)\in\mathbb{N}$ are the {graded Betti numbers} of $M$; equivalently, they record the graded pieces of the derived functor $\operatorname{Tor}$:
\[
\beta_{i,j}(M)=\dim_k\!\bigl(\operatorname{Tor}^{S}_i(M,k)_j\bigr).
\]
Aggregating over all internal degrees yields the {total} $i$–th Betti number,
\[
\beta_i(M):=\sum_{j\in\mathbb{Z}}\beta_{i,j}(M),
\]
which measures the overall size of the $i$–th step in a minimal resolution.

%\smallskip

%\noindent\textbf{Multigraded refinement and degree coarsening.}
It is often convenient to retain the full multidegree data. Endow $S$ with the standard $\mathbb{Z}^n$–grading by $\deg x_r=\mathbf e_r$, and let $M$ be a finitely generated $\mathbb{Z}^n$–graded $S$–module. The {multigraded Betti numbers} are defined by
\[
\beta_{i,\alpha}(M):=\dim_k\!\bigl(\operatorname{Tor}^{S}_i(M,k)_\alpha\bigr)\qquad(\alpha\in\mathbb{Z}^n).
\]
Passing from the multigraded to the singly graded setting amounts to coarsening by the weight map $|\alpha|:=\alpha_1+\cdots+\alpha_n$, and the two notions are compatible via
\[
\beta_{i,j}(M)=\sum_{|\alpha|=j}\beta_{i,\alpha}(M).
\]
Thus, the graded invariants are obtained by summing multigraded contributions over all multidegrees of the same total weight.

%\smallskip

%\noindent\textbf{Behavior in short exact sequences.}
The interaction of Betti numbers with exact sequences reflects the functoriality of $\operatorname{Tor}$. For a short exact sequence of graded $S$–modules
\[
0\to N\to M\to Q\to 0,
\]
the long exact sequence in $\operatorname{Tor}$, taken degreewise, reads
\[
\cdots\longrightarrow
\operatorname{Tor}^S_i(N,k)_j \longrightarrow
\operatorname{Tor}^S_i(M,k)_j \longrightarrow
\operatorname{Tor}^S_i(Q,k)_j \longrightarrow
\operatorname{Tor}^S_{i-1}(N,k)_j \longrightarrow \cdots
\]
for all $i\ge 0$ and $j\in\mathbb{Z}$. Exactness directly yields the subadditivity
\[
\beta_{i,j}(M)\;\le\;\beta_{i,j}(N)\;+\;\beta_{i,j}(Q)\qquad\text{for all }i,j,
\]
and, after summing over $j$, the corresponding statement for total Betti numbers,
\[
\beta_i(M)\;\le\;\beta_i(N)\;+\;\beta_i(Q)\qquad(i\ge 0).
\]
In the special case where the sequence splits (equivalently, $M\cong N\oplus Q$ as graded modules), the degreewise long exact sequence decomposes, and additivity holds on the nose:
\[
\beta_{i,j}(M)\;=\;\beta_{i,j}(N)\;+\;\beta_{i,j}(Q)\qquad\text{for all }i,j.
\]

Suppose in addition that $M$ is $S$–flat. Then $\operatorname{Tor}_i^S(M,k)=0$ for all $i\ge 1$, and
$-\otimes_S k$ is exact on the given short exact sequence. The long exact sequence in $\operatorname{Tor}$ therefore
breaks, degreewise in $j\in\mathbb{Z}$, into isomorphisms
\[
\operatorname{Tor}_i^S(Q,k)_j \;\cong\; \operatorname{Tor}_{i-1}^S(N,k)_j
\qquad\text{for all } i\ge 1,
\]
hence
\begin{equation*}
	\beta_{i,j}(Q)\;=\;\beta_{i-1,j}(N)\qquad (i\ge 1,\; j\in\mathbb{Z}).
\end{equation*}
Summing over $j$ gives the corresponding identities for total Betti numbers:
\[
\beta_i(Q)=\beta_{i-1}(N)\ (i\ge 1).
\]
All statements remain valid in the multigraded setting. In this case, one needs to replace $j$ by a multidegree $u$.

Graded and multigraded Betti numbers capture the size and distribution of syzygies across degrees; multigraded data refine the singly graded picture, and exact sequences control how these invariants compare across extensions and direct sums. For comprehensive background and foundational results in commutative and homological algebra, we refer the reader to the classical texts~\cite{MillerSturmfels2005, bruns1998cohen, Eisenbud1995_CommutativeAlgebra, Rotman2009_IntroHomologicalAlgebra}.

\subsection{Edge Ideals and the Independence Complex}

We fix a field $k$ and write $S=k[x_1,\dots,x_n]$ for the standard $\mathbb{N}$–graded polynomial ring with $\deg x_i=1$. Throughout, all graphs are finite and simple (no loops, no multiple edges), and all monomial ideals live inside $S$.

\begin{definition}[Edge ideal]
	Let $G=(V,E)$ be a finite simple graph on the vertex set $V=\{1,\dots,n\}$. The {edge ideal} of $G$ is the squarefree monomial ideal
	\[
	I(G)\;:=\; \bigl( \,x_ix_j \;\bigm|\; \{i,j\}\in E \,\bigr) \;\subseteq\; S.
	\]
\end{definition}

The preceding definition simply encodes the edge set of $G$ as degree--$2$ squarefree generators in $S$. By construction, $I(G)$ is generated by squarefree monomials of degree $2$ (one for each edge).

It is convenient to record a basic property of squarefree monomial ideals, which immediately applies to edge ideals.

\begin{proposition}[Squarefree monomial ideals are radical]
	If $I\subseteq S$ is generated by squarefree monomials, then $I$ is radical; equivalently $\sqrt{I}=I$.
\end{proposition}

In particular, every edge ideal admits a canonical description as an intersection of monomial primes.

\begin{corollary}[Prime intersection decomposition, \cite{HerzogHibi2011}]
	For any graph $G$, the edge ideal $I(G)$ is radical. Hence it admits an (irredundant) intersection decomposition as an intersection of prime ideals:
	\[
	I(G)\;=\;\bigcap_{t=1}^r \mathfrak p_t,
	\qquad \text{with each } \mathfrak p_t \text{ a monomial prime of } S.
	\]
\end{corollary}

The decomposition in question admits a transparent graph–theoretic formulation once we pass to the language of vertex covers.

\begin{definition}[Vertex cover and minimal vertex cover]
	A subset $C\subseteq V$ is a {vertex cover} of $G$ if every edge of $G$ contains at least one vertex from $C$. A vertex cover $C$ is {minimal} if no proper subset of $C$ is a vertex cover.
\end{definition}

Intuitively, a vertex cover picks one endpoint from each edge. Passing to ideals,
$\mathfrak p_C:=\langle x_i : i\in C\rangle$ kills exactly those variables and thus contains every quadratic
generator $x_ix_j$ with $\{i,j\}\in E$. Minimality of $C$ corresponds to
primality minimal over $I(G)$. We record this as follows.

\begin{theorem}
	Let $G=(V,E)$ be a finite simple graph with $V=\{1,\dots,n\}$, and let
	$S=k[x_1,\dots,x_n]$. Denote the edge ideal by
	\[
	I(G)\;=\;\langle\, x_ix_j \ : \ \{i,j\}\in E \,\rangle \ \subseteq\ S.
	\]
	For $C\subseteq V$ write $\mathfrak p_C:=\langle x_i : i\in C\rangle$.
	Then a prime ideal $\mathfrak p\subseteq S$ is minimal over $I(G)$ if and only if
	$\mathfrak p=\mathfrak p_C$ for some minimal vertex cover $C$ of $G$.
\end{theorem}

\begin{proof}
	{($\Rightarrow$)} Let $\mathfrak p$ be a prime ideal minimal over $I(G)$.
	Set
	\[
	C\ :=\ \{\, i\in V \ : \ x_i\in \mathfrak p \,\}.
	\]
	We claim that $C$ is a vertex cover. Indeed, if $\{i,j\}\in E$, then
	$x_i x_j\in I(G)\subseteq \mathfrak p$, and since $\mathfrak p$ is prime,
	$x_i\in \mathfrak p$ or $x_j\in \mathfrak p$, i.e.\ $\{i,j\}\cap C\neq\varnothing$.
	Thus $C$ meets every edge, so $C$ is a vertex cover. Consequently
	$I(G)\subseteq \mathfrak p_C$ (because each generator $x_ix_j$ has an endpoint in $C$),
	and by construction $\mathfrak p_C\subseteq \mathfrak p$ (since all $x_i$ with $i\in C$
	lie in $\mathfrak p$). Hence $I(G)\subseteq \mathfrak p_C\subseteq \mathfrak p$.
	Minimality of $\mathfrak p$ over $I(G)$ forces $\mathfrak p=\mathfrak p_C$.
	
	It remains to see that $C$ is {minimal} as a vertex cover. Suppose not; then there
	exists $i_0\in C$ such that $C':=C\setminus\{i_0\}$ is still a vertex cover. But then
	$I(G)\subseteq \mathfrak p_{C'}\subsetneq \mathfrak p_C=\mathfrak p$, contradicting the minimality of
	$\mathfrak p$ over $I(G)$. Thus $C$ is a minimal vertex cover.
	
	\smallskip
	{($\Leftarrow$)} Let $C\subseteq V$ be a minimal vertex cover. The ideal
	$\mathfrak p_C=\langle x_i : i\in C\rangle$ is a monomial prime of $S$.
	Since $C$ meets every edge, each generator $x_ix_j$ of $I(G)$ has $i\in C$ or $j\in C$,
	so $x_ix_j\in \mathfrak p_C$; hence $I(G)\subseteq \mathfrak p_C$.
	
	To prove minimality, let $\mathfrak q$ be a prime ideal with
	$I(G)\subseteq \mathfrak q \subseteq \mathfrak p_C$. Set
	$D:=\{\, i\in V : x_i\in \mathfrak q \,\}$. As above, for each edge $\{i,j\}\in E$ we have
	$x_i x_j\in I(G)\subseteq \mathfrak q$, so by primality $x_i\in \mathfrak q$ or $x_j\in \mathfrak q$,
	whence $D$ is a vertex cover. Moreover, $\mathfrak q\subseteq \mathfrak p_C$ implies $D\subseteq C$.
	By minimality of $C$ as a vertex cover we must have $D=C$. Therefore
	$\mathfrak q=\langle x_i : i\in D\rangle=\mathfrak p_C$, showing that no proper containment
	between $I(G)$ and $\mathfrak p_C$ by a prime is possible. Hence $\mathfrak p_C$ is minimal over $I(G)$.
	
	\smallskip
	Combining the two directions, the minimal primes of $I(G)$ are precisely the
	$\mathfrak p_C$ with $C$ a minimal vertex cover of $G$.
\end{proof}

%\subsection*{Stanley--Reisner viewpoint and the independence complex}

It is often fruitful and natural to pass from edge ideals to simplicial complexes because of the
Stanley–Reisner correspondence: every squarefree monomial ideal
$I\subseteq S=k[x_1,\dots,x_n]$ is the Stanley–Reisner ideal $I_\Delta$ of a unique
simplicial complex $\Delta$ on $[n]=\{1,2,\cdots,n\}$, with
\[
I_\Delta \;=\; \big\langle x^\sigma \ :\ \sigma\notin \Delta \big\rangle.
\]
For a graph $G$ on $[n]$ the canonical choice of $\Delta$ is the
{independence complex} $\operatorname{Ind}(G)$, whose faces are the independent
sets of $G$.  This identification transports algebraic invariants of $S/I(G)$ to topological/combinatorial
invariants of $\operatorname{Ind}(G)$, with the precise correspondence given by Hochster’s formula.

\begin{definition}[Independence complex]
	The {independence complex} (or stable set complex) of $G$ is the simplicial complex
	\[
	\mathrm{Ind}(G)\;:=\;\bigl\{\, \sigma\subseteq V \ \bigm|\ \text{$\sigma$ contains no edge of $G$}\,\bigr\}.
	\]
	Equivalently, the faces of $\mathrm{Ind}(G)$ are independent sets in $G$.
\end{definition}

This dictionary identifies edge ideals as Stanley--Reisner ideals.

\begin{proposition}[Edge ideals as Stanley--Reisner ideals, \cite{VanTuyl2013_BeginnerGuideEdgeCoverIdeals}]
	Let $I_\Delta$ denote the Stanley--Reisner ideal of a simplicial complex $\Delta$ on $V$, i.e.
	\[
	I_\Delta \;=\; \bigl( x_F \mid F\subseteq V,\ F\notin \Delta \bigr)
	\quad \text{with } x_F := \prod_{i\in F} x_i.
	\]
	Then for any graph $G$,
	\[
	I(G)\;=\; I_{\mathrm{Ind}(G)}.
	\]
	In particular, the minimal nonfaces of $\mathrm{Ind}(G)$ are exactly the edges of $G$, and the minimal monomial generators of $I_{\mathrm{Ind}(G)}$ are the edge monomials of $I(G)$.
\end{proposition}

%\subsection*{Hochster's formula for graded Betti numbers}

We finally set notation for later homological computations. Let $\Delta$ be a simplicial complex on $V=\{1,\dots,n\}$ and write $S/I_\Delta$ for its Stanley--Reisner ring. For $W\subseteq V$, denote by
\[
\Delta_W \;:=\; \{\,\sigma\in\Delta \mid \sigma\subseteq W\,\}
\]
the induced subcomplex on $W$. It is well known from Hochster’s formula that the graded and multigraded Betti numbers of $S/I_\Delta$ are governed by the (co)homology of these induced subcomplexes \cite{hochster1969prime}. Let $\{\mathbf e_1,\dots,\mathbf e_n\}$ be the standard basis of $\mathbb{Z}^n$.
For $W\subseteq V$, define its indicator multidegree $\mathbf{1}_W\in\mathbb{Z}^n$ by
\[
\mathbf{1}_W \;:=\; \sum_{i\in W} \mathbf e_i .
\]
We will use this perspective in subsequent sections without further comment.

\begin{theorem}[Hochster's formula (multigraded form), \cite{bruns1998cohen}]
	Let $S=k[x_1,\dots,x_n]$ with its $\mathbb{Z}^n$–grading $\deg(x_i)=\mathbf{e}_i$, and let
	$\Delta$ be a simplicial complex on $[n]=\{1,\dots,n\}$. Write $I_\Delta$ for the
	Stanley--Reisner ideal and $k[\Delta]=S/I_\Delta$.
	For $W\subseteq[n]$, let $\mathbf{1}_W\in\{0,1\}^n$ denote its indicator multidegree, and
	set $\Delta_W:=\{\sigma\in\Delta\mid \sigma\subseteq W\}$.
	Then for all $i\ge 0$ and $W\subseteq[n]$,
	\[
	\beta_{i,\mathbf{1}_W}\bigl(k[\Delta]\bigr)
	\;=\;
	\dim_k\!\Bigl(\operatorname{Tor}_i^S\bigl(k[\Delta],k\bigr)_{\mathbf{1}_W}\Bigr)
	\;=\;
	\dim_k\,\widetilde H^{\,|W|-i-1}\!\bigl(\Delta_W;k\bigr).
	\]
\end{theorem}

\begin{proof}
	Let $K_\bullet=K_\bullet(x_1,\dots,x_n;S)$ be the Koszul complex on $(x_1,\dots,x_n)$.
	Write $K_i\cong \bigwedge^i V\otimes_k S$ where $V=\bigoplus_{j=1}^n k\,e_j$.
	For $F=\{j_1<\cdots<j_i\}$ set $e_F:=e_{j_1}\wedge\cdots\wedge e_{j_i}$.
	The differential is
	\[
	d\bigl(e_F\otimes f\bigr)
	=\sum_{t=1}^i (-1)^{t-1}\,e_{F\setminus\{j_t\}}\otimes x_{j_t}f.
	\]
	Give $K_\bullet$ the $\mathbb{Z}^n$–grading by $\deg(e_j)=\mathbf{e}_j$, so
	$\deg(e_F\otimes x^{\alpha})=\sum_{j\in F}\mathbf{e}_j+\alpha$ and $d$ is homogeneous of
	multidegree $\mathbf{0}$. Since $K_\bullet$ is a free resolution of $k=S/(\mathbf{x})$,
	\[
	\operatorname{Tor}_i^S\bigl(k[\Delta],k\bigr)\ \cong\ H_i\!\bigl(k[\Delta]\otimes_S K_\bullet\bigr).
	\]
	
	Fix $W\subseteq[n]$. The homogeneous component
	$\bigl(k[\Delta]\otimes_S K_i\bigr)_{\mathbf{1}_W}$ is spanned by the classes of
	\[
	e_F \otimes \overline{x^{U}}
	\qquad\text{with } F\subseteq[n],\ |F|=i,\ U\subseteq[n],\ F\cap U=\varnothing,\ F\cup U=W,
	\]
	where $\overline{(\cdot)}$ is the image in $k[\Delta]=S/I_\Delta$ and
	$x^{U}:=\prod_{u\in U}x_u$. The multidegree condition is
	$\deg(e_F)+\deg(x^{U})=\mathbf{1}_W$, i.e.\ $F\cap U=\varnothing$ and $F\cup U=W$.
	
	In $k[\Delta]$, $\overline{x^{U}}\neq 0$ iff $U\in\Delta$. Therefore a $k$–basis of
	$\bigl(k[\Delta]\otimes_S K_i\bigr)_{\mathbf{1}_W}$ is
	\[
	\mathcal{B}_{i,W}
	=\bigl\{\, e_F\otimes \overline{x^{W\setminus F}}\ :\ F\subseteq W,\ |F|=i,\ W\setminus F\in\Delta\,\bigr\}.
	\]
	Equivalently, with $\sigma:=W\setminus F$, we index by faces $\sigma\in\Delta_W$ with
	$|\sigma|=|W|-i$.
	
	For $F\subseteq W$ and $\sigma=W\setminus F$,
	\[
	d\bigl(e_F\otimes \overline{x^\sigma}\bigr)
	=\sum_{j\in F} (-1)^{\operatorname{pos}(j;F)-1}\,
	e_{F\setminus\{j\}}\otimes \overline{x_j x^\sigma}
	=\sum_{j\in F} (-1)^{\operatorname{pos}(j;F)-1}\,
	e_{F\setminus\{j\}}\otimes \overline{x^{\sigma\cup\{j\}}}.
	\]
	A summand survives precisely when $\sigma\cup\{j\}\in\Delta$.
	For $i\ge 0$ and $F\subseteq W$ with $|F|=i$, set $\sigma:=W\setminus F$ and define
	\[
	\Phi_W^{\,i}:\ \bigl(k[\Delta]\otimes_S K_i\bigr)_{\mathbf{1}_W}\longrightarrow
	C^{\,|W|-i-1}(\Delta_W;k),\qquad
	\Phi_W^{\,i}\bigl(e_F\otimes \overline{x^{\,\sigma}}\bigr)
	\;:=\;
	(-1)^{\binom{|\sigma|}{2}}\,
	\operatorname{sgn}(\pi_{F,\sigma})\ \sigma^{\ast}.
	\]
	Here $\pi_{F,\sigma}$ is the unique permutation that reorders the concatenation
	$(\sigma, F)$ into the (increasing) order of $W$, and $\operatorname{sgn}(\cdot)$ is its sign.
	This is a well–defined $k$–isomorphism on the $\mathbf{1}_W$–multigraded piece. The simplicial
	coboundary $\delta:C^{q}(\Delta_W;k)\to C^{q+1}(\Delta_W;k)$ is taken with the standard sign
	convention:
	\[
	\delta(\sigma)\;=\;\sum_{\substack{\tau\in\Delta_W\\ \sigma\subset\tau,\ |\tau|=|\sigma|+1}}
	(-1)^{\operatorname{pos}(v;\tau)-1}\,\tau,
	\qquad \text{for }\tau=\sigma\cup\{v\},
	\]
	where $\operatorname{pos}(v;\tau)$ is the position of $v$ in the increasing ordering of $\tau$.

	For $j\in F$, put $\tau_j:=\sigma\cup\{j\}$.
	Tracing the basis element $e_F\otimes\overline{x^{\,\sigma}}$ around the square
	\[
	\begin{tikzcd}[column sep=large]
		\cdots \arrow[r,"d"] &
		\bigl(k[\Delta]\!\otimes\! K_{i}\bigr)_{\mathbf{1}_W}
		\arrow[r,"d"] \arrow[d,"\Phi_W^{i}"',"\cong"] &
		\bigl(k[\Delta]\!\otimes\! K_{i-1}\bigr)_{\mathbf{1}_W}
		\arrow[r,"d"] \arrow[d,"\Phi_W^{i-1}"',"\cong"] &
		\cdots \\[-1ex]
		\cdots \arrow[r,"\delta"'] &
		C^{\,|W|-i-1}(\Delta_W;k)
		\arrow[r,"\delta"'] &
		C^{\,|W|-i}(\Delta_W;k)
		\arrow[r,"\delta"'] &
		\cdots
	\end{tikzcd}
	\]
	we compare the coefficient of $\tau_j^\ast$ obtained by going “right then down” versus “down then right.”
	The Koszul differential contributes the factor $(-1)^{\operatorname{pos}(j;F)-1}$ when deleting $j$ from $F$,
	and the simplicial coboundary contributes $(-1)^{\operatorname{pos}(j;\tau_j)-1}$ when adding $j$ to $\sigma$.
	Thus the two routes yield, respectively,
	\[
	\text{(right then down)}\quad
	(-1)^{\operatorname{pos}(j;F)-1}\cdot
	(-1)^{\binom{|\tau_j|}{2}}\cdot
	\operatorname{sgn}\!\bigl(\pi_{F\setminus\{j\},\,\tau_j}\bigr),
	\]
	\[
	\text{(down then right)}\quad
	(-1)^{\binom{|\sigma|}{2}}\cdot
	\operatorname{sgn}\!\bigl(\pi_{F,\sigma}\bigr)\cdot
	(-1)^{\operatorname{pos}(j;\tau_j)-1}.
	\]
	
	Two elementary identities account for the sign reconciliation:
	\begin{align}
		\operatorname{sgn}\!\bigl(\pi_{F\setminus\{j\},\,\tau_j}\bigr)
		&=
		(-1)^{\operatorname{pos}(j;F)-1}\,
		(-1)^{\operatorname{pos}(j;\tau_j)-1}\,
		(-1)^{|\sigma|}\,
		\operatorname{sgn}\!\bigl(\pi_{F,\sigma}\bigr),
		\label{eq:perm-sign}\\
		\binom{|\tau_j|}{2}
		&=\binom{|\sigma|+1}{2}
		\equiv \binom{|\sigma|}{2}+|\sigma|\pmod 2.
		\label{eq:binom-parity}
	\end{align}
	Using \eqref{eq:perm-sign} and \eqref{eq:binom-parity}, the “right then down” coefficient equals
	\[
	(-1)^{\binom{|\tau_j|}{2}+|\sigma|}\,
	(-1)^{2(\operatorname{pos}(j;F)-1)}\,
	\operatorname{sgn}(\pi_{F,\sigma})\,
	(-1)^{\operatorname{pos}(j;\tau_j)-1}
	=
	(-1)^{\binom{|\sigma|}{2}}\,
	\operatorname{sgn}(\pi_{F,\sigma})\,
	(-1)^{\operatorname{pos}(j;\tau_j)-1},
	\]
	which is precisely the “down then right” coefficient. Hence
	\[
	\Phi_W^{\,i-1}\circ d\;=\;\delta\circ \Phi_W^{\,i}.
	\]

	The preceding identity shows that, on the multidegree $\mathbf{1}_W$,
	\[
	\bigl(k[\Delta]\otimes_S K_\bullet\bigr)_{\mathbf{1}_W}
	\ \cong\ C^{\,|W|-1-\bullet}(\Delta_W;k)
	\quad\text{as chain complexes.}
	\]
	Therefore
	\[
	H_i\!\bigl((k[\Delta]\otimes_S K_\bullet)_{\mathbf{1}_W}\bigr)
	\ \cong\ \widetilde H^{\,|W|-i-1}(\Delta_W;k),
	\]
	and, taking $k$–dimensions of multigraded pieces,
	\[
	\beta_{i,\mathbf{1}_W}\bigl(k[\Delta]\bigr)
	=\dim_k \operatorname{Tor}_i^S\bigl(k[\Delta],k\bigr)_{\mathbf{1}_W}
	=\dim_k \widetilde H^{\,|W|-i-1}(\Delta_W;k).\qedhere
	\]

\end{proof}

In particular, for a simple graph $G$ on vertex set $[n]$, Hochster’s formula gives,
for every $W\subseteq [n]$ and $i\ge 0$,
\[
\beta_{i,\mathbf{1}_W}\!\bigl(S/I(G)\bigr)
\;=\;
\dim_k \widetilde H^{\,|W|-i-1}\!\bigl(\operatorname{Ind}(G)_W;k\bigr),
\]
where $\operatorname{Ind}(G)_W$ denotes the subcomplex of the independence complex
$\operatorname{Ind}(G)$ induced on $W$.

From the short exact sequence $0\to I(G)\to S\to S/I(G)\to 0$ and the flatness of $S$ over itself,
the long exact sequence in $\operatorname{Tor}$ yields degreewise isomorphisms
\[
\operatorname{Tor}_i^S\!\bigl(S/I(G),k\bigr)_{\mathbf{1}_W}
\;\cong\;
\operatorname{Tor}_{i-1}^S\!\bigl(I(G),k\bigr)_{\mathbf{1}_W}
\qquad\text{for all }i\ge 1\text{ and }W\subseteq[n].
\]
Hence the multigraded Betti numbers satisfy
\[
\beta_{i-1,\mathbf{1}_W}\!\bigl(I(G)\bigr)
\;=\;
\beta_{i,\mathbf{1}_W}\!\bigl(S/I(G)\bigr)
\qquad(i\ge 1,\ W\subseteq[n]).
\]
In addition, $\beta_{0,\mathbf{0}}(S/I(G))=1$ and $\beta_{0,u}(S/I(G))=0$ for $u\neq \mathbf{0}$. On the other hand, for the edge ideal $I(G)$ one has
\[
\beta_{0,j}\!\bigl(I(G)\bigr)=
\begin{cases}
	|E(G)|, & \text{if } j=2,\\[2pt]
	0, & \text{if } j\neq 2,
\end{cases}
\]
since $\beta_{0,j}(I)$ counts the number of minimal generators of $I$ in degree $j$, and
$I(G)$ is minimally generated by the quadrics $\{x_ix_j:\{i,j\}\in E(G)\}$.
Equivalently, in the multigraded form:
$\beta_{0,\mathbf e_i+\mathbf e_j}\!\bigl(I(G)\bigr)=1$ if $\{i,j\}\in E(G)$ and $0$ otherwise.

\begin{example}[Path on three vertices]\label{ex:P3}
	Let $G=P_3$ with $V=\{1,2,3\}$ and $E=\{\{1,2\},\{2,3\}\}$. Then
	\[
	I(G)\;=\;(x_1x_2,\;x_2x_3)\;\subseteq\;S=k[x_1,x_2,x_3].
	\]
	The minimal vertex covers are $C_1=\{2\}$ and $C_2=\{1,3\}$. Therefore
	\[
	I(G)\;=\;(x_2)\ \cap\ (x_1,x_3).
	\]
	The independence complex is
	\[
	\mathrm{Ind}(G)=\bigl\{\varnothing,\{1\},\{2\},\{3\},\{1,3\}\bigr\},
	\]
	whose {facets} (inclusion–maximal faces) are $\{1,3\}$ and $\{2\}$.
	Applying Hochster’s formula to $\Delta=\mathrm{Ind}(G)$ gives, for the graded Betti numbers of $S/I(G)$,
	\[
	\beta_{0,0}=1,\qquad \beta_{1,2}=2,\qquad \beta_{2,3}=1,
	\]
	and $\beta_{i,j}=0$ for all other pairs $(i,j)$. Equivalently, the only nonzero entries occur in homological degrees $i=0,1,2$ and internal degrees $j=0,2,3$, at the positions listed above.
	With rows indexed by internal degree $j$ and columns by homological degree $i$:
	\[
	\begin{array}{c|ccc}
		& i=0 & i=1 & i=2\\ \hline
		j=0 & 1 & 0 & 0\\
		j=1 & 0 & 0 & 0\\
		j=2 & 0 & 2 & 0\\
		j=3 & 0 & 0 & 1
	\end{array}
	\]
	The resolution is
	\[
	0 \longleftarrow I(G)
	\longleftarrow S(-2)^2
	\longleftarrow S(-3)
	\longleftarrow 0.
	\]
\end{example}

\begin{example}[Star graphs and the star ideal]\label{ex:star}
	Let $G$ be a star on $t+1$ vertices with center $x$ and leaves $N(x)=\{x_1,\dots,x_t\}$, i.e.,
	\[
	E(G)=\bigl\{\{x,x_1\},\dots,\{x,x_t\}\bigr\}.
	\]
	Then the edge ideal is the {star ideal}
	\[
	I(G)\;=\;\langle\, x x_1,\dots, x x_t\,\rangle \;=\; x\,\langle x_1,\dots,x_t\rangle \;\subseteq\; S=k[x,x_1,\dots,x_t].
	\]
	A vertex cover is either $\{x\}$ or the full leaf set $\{x_1,\dots,x_t\}$; both are minimal. Consequently,
	\[
	I(G)\;=\;(x)\ \cap\ (x_1,\dots,x_t),
	\]
	We have
	\[
	\mathrm{Ind}(G)\;=\;\bigl(\text{simplex on }\{x_1,\dots,x_t\}\bigr)\ \sqcup\ \{x\},
	\]
	i.e., the disjoint union of a $(t-1)$–simplex and an isolated vertex.
	Write $S/I(G)$ for the Stanley–Reisner ring of $\mathrm{Ind}(G)$. Since every induced subcomplex on a subset $W$ that contains $x$ and at least one leaf has exactly two connected components, Hochster’s formula yields:
	\[
	\beta_{0,0}=1,\qquad
	\beta_{j-2,\,j}=\binom{t}{\,j-1\,}\ \ \text{for }2\le j\le t+1,\qquad
	\beta_{i,j}=0\ \text{otherwise.}
	\]
	Equivalently,
	\[
	\beta_{i,\,i+2}=\binom{t}{\,i+1\,}\quad\text{for }0\le i\le t-1,
	\]
	with all remaining graded Betti numbers vanishing. Thus the nonzero entries lie on the single diagonal $j=i+2$ (together with $\beta_{0,0}$), so the resolution is $2$–linear.
	In particular, the minimal graded free resolution has the form
	\[
	0 \longleftarrow I(G) 
	\longleftarrow S(-2)^{\binom{t}{1}}
	\longleftarrow S(-3)^{\binom{t}{2}}
	\longleftarrow \cdots
	\longleftarrow S\bigl(-(t+1)\bigr)^{\binom{t}{t}}
	\longleftarrow 0.
	\]
\end{example}

\begin{example}[Complete graph]\label{ex:Kn}
	Let $G=K_n$. Then
	\[
	I(G)\;=\;\bigl(x_ix_j \mid 1\le i<j\le n\bigr)\ \subseteq\ S=k[x_1,\dots,x_n].
	\]
	A minimal vertex cover has size $n-1$, and in fact {every} $(n-1)$–subset of $V$ is a minimal vertex cover. Consequently,
	\[
	I(K_n)\;=\;\bigcap_{\substack{C\subseteq V\\ |C|=n-1}} (x_i \mid i\in C),
	\qquad
	\operatorname{ht} I(K_n)=n-1.
	\]
	Here $\mathrm{Ind}(G)$ consists precisely of the empty face and the singletons:
	\[
	\mathrm{Ind}(K_n)=\bigl\{\varnothing,\{1\},\dots,\{n\}\bigr\},
	\]
	so every induced subcomplex $\mathrm{Ind}(K_n)_W$ is a discrete set of $|W|$ points. Hochster’s formula therefore yields a linear (pure) pattern for the Betti numbers of $S/I(K_n)$:
	\[
	\beta_{0,0}=1,\qquad
	\beta_{j-1,\,j}=\binom{n}{j}\,(j-1)\ \text{ for } 2\le j\le n,\qquad
	\beta_{i,j}=0\ \text{ otherwise.}
	\]
	Equivalently,
	\[
	\beta_{i,\,i+1}= i\,\binom{n}{\,i+1\,}\quad\text{for }1\le i\le n-1,
	\]
	and all remaining graded Betti numbers vanish. In particular, the nonzero entries are supported in internal degrees $j=0,2,3,\dots,n$, with a single diagonal $j=i+1$ (besides $\beta_{0,0}$).
	The minimal graded free resolution is pure and $1$–linear after the initial step:
	\[
	0 \longleftarrow I(K_n)
	\longleftarrow S(-2)^{\binom{n}{2}}
	\longleftarrow S(-3)^{2\binom{n}{3}}
	\longleftarrow \cdots
	\longleftarrow S(-n)^{(n-1)\binom{n}{n}}
	\longleftarrow 0.
	\]
\end{example}

Edge ideals often admit a special splitting that allow one to compute graded Betti numbers recursively. The guiding notion is the following.

\begin{definition}[Betti splitting, \cite{VanTuyl2013_BeginnerGuideEdgeCoverIdeals, Francisco2009_SplittingsOfMonomialIdeals}]
	Let $I\subseteq S$ be a monomial ideal with a decomposition
	\[
	I \;=\; J \;+\; K
	\]
	into monomial ideals $J,K\subseteq S$ such that the sets of minimal monomial generators are disjoint,
	$\mathcal{G}(I)=\mathcal{G}(J)\sqcup \mathcal{G}(K)$. We say that $I=J+K$ is a {Betti splitting} if, for all $i\ge 0$ and all $j\in\mathbb{Z}$,
	\[
	\beta_{i,j}(I)\;=\;\beta_{i,j}(J)\;+\;\beta_{i,j}(K)\;+\;\beta_{i-1,j}\bigl(J\cap K\bigr),
	\]
	with the convention $\beta_{-1,j}(\,\cdot\,)=0$. Equivalently, the same identity holds in the multigraded setting. 
\end{definition}

\noindent
In practice, one verifies the Betti splitting property via well–known sufficient criteria (e.g., vanishing of the natural maps on $\operatorname{Tor}$ or linear–quotient hypotheses on $J$ and $K$); see standard references for details. The point is that a Betti splitting reduces the Betti numbers of $I$ to those of $J$, $K$, and their intersection $J\cap K$.

\medskip

\noindent
We now record the vertex decomposition for edge ideals, which provides a canonical recursive step \cite{VanTuyl2013_BeginnerGuideEdgeCoverIdeals, Francisco2009_SplittingsOfMonomialIdeals}.

\begin{theorem}[ \cite{VanTuyl2013_BeginnerGuideEdgeCoverIdeals, Francisco2009_SplittingsOfMonomialIdeals}]
	Let $G$ be a finite simple graph on $V=\{1,\dots,n\}$, fix a vertex $x\in V$, and assume that $G\setminus\{x\}$ is not a graph of isolated vertices. Write the (open) neighborhood of $x$ as $N(x)=\{x_1,\dots,x_t\}$. Set
	\[
	J \;:=\; \langle\, x x_1, \dots, x x_t \,\rangle
	\qquad\text{and}\qquad
	K \;:=\; I\bigl(G\setminus\{x\}\bigr).
	\]
	Then
	\[
	I(G)\;=\;J\;+\;K
	\]
	is a Betti splitting of $I(G)$. Moreover, the intersection decomposes explicitly as
	\[
	J\cap K
	\;=\;
	\langle x x_1,\dots, x x_t\rangle \cap I\bigl(G\setminus\{x\}\bigr)
	\;=\;
	x\, I\bigl(G(x)\bigr)\;+\; x x_1\, I(G_1)\;+\;\cdots\;+\; x x_t\, I(G_t),
	\]
	where, for any vertex $a\in V$, the auxiliary graph $G(a)$ is defined on $V\setminus\{a\}$ by
	\[
	E\bigl(G(a)\bigr)
	\;:=\;
	\bigl\{\, \{u,v\}\in E(G) \ \bigm|\ \{u,v\}\cap N(a)\neq\emptyset,\ u\neq a,\ v\neq a \,\bigr\},
	\]
	and $G_i:=G\setminus (N(x)\cup N(x_i))$ for $i=1,\dots,t$.
\end{theorem}

\noindent
Consequently, the graded Betti numbers satisfy the recursive identity
\[
\beta_{i,j}\bigl(I(G)\bigr)
\;=\;
\beta_{i,j}(J)
\;+\;
\beta_{i,j}\!\bigl(I(G\setminus\{x\})\bigr)
\;+\;
\beta_{i-1,j}\!\Bigl( x\, I\bigl(G(x)\bigr)\;+\;\sum_{r=1}^t x x_r\, I(G_r) \Bigr),
\]
and similarly in the multigraded setting. This vertex–based splitting is standard in the literature and underlies many inductive computations of Betti tables for edge ideals \cite{ha2008monomial}.

\section{Persistent graded Betti numbers of modules and monomial ideals}

The goal of this section is to extend familiar graded Betti invariants to a functorial,
two–time–point setting. Throughout, a filtration $(M(t))_{t\in T}$ (or $(\Delta_t)_{t\in T}$)
comes equipped with degree–preserving structure maps for $a\le b$, and every definition below
reduces to the usual invariant when $a=b$.

\begin{definition}[Persistent graded Betti numbers]\label{def:persistent-betti}
	Let $\{M(t)\}_{t\in T}$ be a filtration of graded $S$–modules indexed by a totally ordered set $(T,\le)$, and suppose for $a\le b$ we are given structure maps $M(a)\to M(b)$ that are degree-preserving $S$–homomorphisms.
	For $i,j\in\mathbb{Z}$ define the {persistent graded Betti numbers} by
	\[
	\beta_{i,j}^{\,a, b}(M(\bullet))
	\;:=\;
	\operatorname{rank}\!\left(
	\operatorname{Tor}_i^S\bigl(k,M(a)\bigr)_j
	\longrightarrow
	\operatorname{Tor}_i^S\bigl(k,M(b)\bigr)_j
	\right),
	\qquad a\le b.
	\]
	When $a=b$ this recovers the usual Betti numbers: $\beta_{i,j}^{\,a\to a}(M(\bullet))=\beta_{i,j}\bigl(M(a)\bigr)$.
	In the multigraded case, replace $j$ by $\alpha\in\mathbb{Z}^n$:
	\[
	\beta_{i,\alpha}^{\,a, b}(M(\bullet))
	\;:=\;
	\operatorname{rank}\!\left(
	\operatorname{Tor}_i^S\bigl(k,M(a)\bigr)_\alpha
	\longrightarrow
	\operatorname{Tor}_i^S\bigl(k,M(b)\bigr)_\alpha
	\right).
	\]
\end{definition}

\noindent
Definition~\ref{def:persistent-betti} packages the classical graded Betti numbers into a
rank of the comparison map induced by $M(a)\to M(b)$. In particular, taking $a=b$
recovers $\beta_{i,\bullet}(M(a))$, so all results that follow are genuine
persistent extensions of the standard theory.

\noindent
The next statement is the persistent counterpart of Hochster’s formula: it identifies
the ranks in Definition~\ref{def:persistent-betti} for Stanley–Reisner rings with the
ranks of cohomology restriction maps along $(\Delta_a)_W\hookrightarrow(\Delta_b)_W$—and
when $a=b$ it collapses to the classical Hochster isomorphism degreewise.

\begin{theorem}[Persistent Hochster formula (multigraded)]\label{thm:persistent-hochster}
	Let $S=k[x_1,\dots,x_n]$ with $\deg(x_i)=\mathbf e_i$ and let
	$\{\Delta_t\}_{t\in T}$ be a filtration of simplicial complexes on $[n]$
	with inclusions $\Delta_a\subseteq \Delta_b$ for $a\le b$.
	Write $I_{\Delta_t}$ for the Stanley--Reisner ideal and $k[\Delta_t]=S/I_{\Delta_t}$.
	For $W\subseteq[n]$, set $(\Delta_t)_W:=\{\sigma\in\Delta_t:\sigma\subseteq W\}$.
	Then for all $i\ge 0$ and $W\subseteq[n]$,
	\[
	\begin{aligned}
		\beta^{\,a, b}_{i,\mathbf{1}_W}\!\bigl(k[\Delta_\bullet]\bigr)
		&:= \operatorname{rank}\!\Bigl(
		\operatorname{Tor}_i^S\bigl(k[\Delta_b],k\bigr)_{\mathbf{1}_W}
		\to
		\operatorname{Tor}_i^S\bigl(k[\Delta_a],k\bigr)_{\mathbf{1}_W}
		\Bigr)
		\\[2pt]
		&=\ \operatorname{rank}\!\Bigl(
		\widetilde H^{\,|W|-i-1}\!\bigl((\Delta_b)_W;k\bigr)
		\to
		\widetilde H^{\,|W|-i-1}\!\bigl((\Delta_a)_W;k\bigr)
		\Bigr).
	\end{aligned}
	\]
\end{theorem}

\begin{proof}[Idea of proof]
	By Hochster’s isomorphism,
	$\operatorname{Tor}_i^S(k[\Delta_t],k)_{\mathbf{1}_W}\cong
	\widetilde H^{\,|W|-i-1}((\Delta_t)_W;k)$ for each $t$ and $W$.
	If $a\le b$, the inclusion $\Delta_a\subseteq\Delta_b$ gives
	$I_{\Delta_b}\subseteq I_{\Delta_a}$ and hence a {surjection}
	$k[\Delta_b]\twoheadrightarrow k[\Delta_a]$.
	Functoriality of $\operatorname{Tor}$ in the first variable yields the map
	\[
	\operatorname{Tor}_i^S\bigl(k[\Delta_b],k\bigr)_{\mathbf{1}_W}\longrightarrow
	\operatorname{Tor}_i^S\bigl(k[\Delta_a],k\bigr)_{\mathbf{1}_W},
	\]
	which corresponds, under Hochster’s identification and the inclusion
	$(\Delta_a)_W\subseteq(\Delta_b)_W$, to the restriction map in cohomology
	\[
	\widetilde H^{\,|W|-i-1}\!\bigl((\Delta_b)_W;k\bigr)\longrightarrow
	\widetilde H^{\,|W|-i-1}\!\bigl((\Delta_a)_W;k\bigr),
	\]
	according to the below diagrams.
	% For a fixed W ⊆ [n], the map induced by k[Δ] → k[Δ′] becomes cochain restriction
	\[
	\begin{tikzcd}[column sep=huge]
		\bigl(k[\Delta_b]\!\otimes\! K_{i}\bigr)_{\mathbf{1}_W}
		\arrow[r,"d"] \arrow[d] &
		\bigl(k[\Delta_b]\!\otimes\! K_{i-1}\bigr)_{\mathbf{1}_W}
		\arrow[d] \\
		\bigl(k[\Delta_a]\!\otimes\! K_{i}\bigr)_{\mathbf{1}_W}
		\arrow[r,"d"] &
		\bigl(k[\Delta_a]\!\otimes\! K_{i-1}\bigr)_{\mathbf{1}_W}
	\end{tikzcd}
	\longleftrightarrow
	\begin{tikzcd}[column sep=huge]
		C^{\,|W|-i-1}((\Delta_b)_W;k)
		\arrow[r,"\delta"] \arrow[d,"i_W^{*}"'] &
		C^{\,|W|-i}((\Delta_b)_W;k)
		\arrow[d,"i_W^{*}"'] \\
		C^{\,|W|-i-1}((\Delta_a)_W;k)
		\arrow[r,"\delta"'] &
		C^{\,|W|-i}((\Delta_a)_W;k)
	\end{tikzcd}
	\]
	% Left square is induced by the ring map; right square is restriction along i_W: Δ′_W ↪ Δ_W.
	% The Φ_W isomorphisms identify these squares.
	Taking ranks gives the stated equality.

\end{proof}

\noindent
Summing the multigraded statement over all $W\subseteq[n]$ with $|W|=j$ yields the
$\mathbb{Z}$–graded version: the persistent Betti number in internal degree $j$
is the sum of the ranks of the induced maps on $\widetilde H^{\,j-i-1}((\Delta_t)_W;k)$.

\begin{corollary}[Persistent Hochster, $\mathbb{Z}$–graded]
	For the $\mathbb{Z}$–graded Betti numbers (internal degree $j$),
	\[
	\beta^{\,a, b}_{i,j}\!\bigl(k[\Delta_\bullet]\bigr)
	\;=\;
	\sum_{\substack{W\subseteq[n]\\ |W|=j}}
	\operatorname{rank}\!\Bigl(
	\widetilde H^{\,j-i-1}\!\bigl((\Delta_b)_W;k\bigr)
	\longrightarrow
	\widetilde H^{\,j-i-1}\!\bigl((\Delta_a)_W;k\bigr)
	\Bigr),
	\qquad \Delta_a\subseteq \Delta_b.
	\]
\end{corollary}

\noindent
Over a field, cohomology is the linear dual of homology. We will use the following
rank duality to freely switch between cohomological and homological formulations of
the persistent ranks.

\begin{lemma}[Cohomology--homology rank duality over a field]
	Let $k$ be a field, $a\le b$, $W\subseteq[n]$, and put $q:=|W|-i-1$. Then
	\[
	\mathrm{rank}\!\Bigl(
	\widetilde H^{\,q}\!\bigl((\Delta_b)_W;k\bigr)\to
	\widetilde H^{\,q}\!\bigl((\Delta_a)_W;k\bigr)
	\Bigr)
	=
	\mathrm{rank}\!\Bigl(
	\widetilde H_{\,q}\!\bigl((\Delta_a)_W;k\bigr)\to
	\widetilde H_{\,q}\!\bigl((\Delta_b)_W;k\bigr)
	\Bigr).
	\]
\end{lemma}

\begin{proof}
	Over a field $k$, the simplicial cochain complex is the $k$–linear dual of the chain complex:
	$C^q(\Gamma;k)=\mathrm{Hom}_k(C_q(\Gamma;k),k)$ with coboundary
	$\delta^q=(\partial_{q+1})^\ast$. Hence
	$\widetilde H^{\,q}(\Gamma;k)\cong \mathrm{Hom}_k(\widetilde H_{\,q}(\Gamma;k),k)$
	naturally in $\Gamma$. Apply this to the inclusion
	$(\Delta_a)_W\hookrightarrow(\Delta_b)_W$; the induced map on cohomology is the dual of
	the induced map on homology. For linear maps of finite-dimensional $k$–vector spaces,
	$\mathrm{rank}(f)=\mathrm{rank}(f^\ast)$. The claim follows.
\end{proof}

Therefore, the definition of the persistent graded Betti numbers in the PSRT framework,
where they were introduced via homological groups, is natural. The lemma above shows that
using homology instead of cohomology yields equivalent and well-defined invariants over a field.

\noindent
The familiar shift $\beta_i(S/I)=\beta_{i-1}(I)$ persists functorially: 
the connecting isomorphisms in the long exact $\operatorname{Tor}$ sequence are natural, 
hence commute with the morphisms induced by $I\to I'$ and $S/I\to S/I'$,
intertwining all structure maps in a filtration. 
Setting $a=b$ recovers the classical (pointwise) shift.

\begin{proposition}[Betti shift between an ideal and its quotient; naturality]
	\label{prop:betti-shift-nat}
	Let \(S\) be a \(\mathbb{Z}\)- or \(\mathbb{Z}^n\)-graded Noetherian ring and \(I\subset S\) a homogeneous ideal.
	Fix a field \(k\) and view \(k\) as a graded \(S\)-module concentrated in degree \(0\).
	Then:
	
	\begin{enumerate}
		\item[\textup{(a)}] \textbf{Betti shift (pointwise).}
		For all \(i\ge 1\) and all degrees \(j\) (resp.\ multidegrees \(\alpha\)),
		\[
		\beta_{i,j}(S/I)=\beta_{i-1,j}(I),\qquad
		\bigl(\text{resp. }\ \beta_{i,\alpha}(S/I)=\beta_{i-1,\alpha}(I)\bigr),
		\]
		and \(\beta_{0,0}(S/I)=1\), \(\beta_{0,j}(S/I)=0\) for \(j\neq 0\)
		{(resp.\ \(\beta_{0,\alpha}(S/I)=\mathbf{1}_{\{\alpha=\mathbf{0}\}}\))}.
		
		\item[\textup{(b)}] \textbf{Naturality with respect to maps of short exact sequences.}
		Let \(\phi:I\to I'\) be a graded \(S\)-linear map and let \(\overline{\phi}:S/I\to S/I'\) be the induced map.
		Then for every \(i\ge 1\) the isomorphisms in \textup{(a)} fit into commutative squares, degreewise:
		\[
		\begin{array}{ccc}
			\operatorname{Tor}^{S}_i(S/I,k)_\bullet
			& \xrightarrow{\ \overline{\phi}_\ast\ } &
			\operatorname{Tor}^{S}_i(S/I',k)_\bullet\\[2pt]
			\big\downarrow\scriptstyle{\cong} & & \big\downarrow\scriptstyle{\cong}\\[2pt]
			\operatorname{Tor}^{S}_{i-1}(I,k)_\bullet
			& \xrightarrow{\ \phi_\ast\ } &
			\operatorname{Tor}^{S}_{i-1}(I',k)_\bullet
		\end{array}
		\]
		where \(\bullet\) denotes either a single degree \(j\) or a multidegree \(\alpha\).
		
		\item[\textup{(c)}] \textbf{Persistent families.}
		If \((I_t)_{t\in T}\) is a graded family indexed by a poset \(T\), and for
		\(s\le t\) we have a morphism \(\phi_{s\to t}:I_s\to I_t\) with induced
		\(\overline{\phi}_{s\to t}:S/I_s\to S/I_t\), then the identifications in \textup{(a)} are
		compatible with the structure maps in the sense of \textup{(b)} for every \(s\le t\).
	\end{enumerate}
\end{proposition}

\begin{proof}
	\textup{(a)} is a common result, is already discussed in the previous sections.
	
	For \textup{(b)}, let \(\phi:I\to I'\) be graded \(S\)-linear and consider the morphism of short exact sequences
	\[
	\begin{array}{ccccccccc}
		0 &\to& I &\xrightarrow{\ \iota\ }& S &\xrightarrow{\ \pi\ }& S/I &\to& 0 \\
		&   & \big\downarrow\scriptstyle{\phi} &  & \big\Vert & & \big\downarrow\scriptstyle{\overline{\phi}} & & \\
		0 &\to& I'&\xrightarrow{\ \iota'\ }& S &\xrightarrow{\ \pi'\ }& S/I'&\to& 0 \ .
	\end{array}
	\]
	Functoriality of \(\operatorname{Tor}\) provides a morphism of long exact sequences, and the connecting isomorphisms
	\(\delta_i: \operatorname{Tor}^{S}_i(S/I,k)_\bullet \xrightarrow{\ \cong\ } \operatorname{Tor}^{S}_{i-1}(I,k)_\bullet\)
	are {natural} \cite{Rotman2009_IntroHomologicalAlgebra}. Concretely, for all \(i\ge 1\) and in every degree \(\bullet\),
	\[
	\begin{array}{ccc}
		\operatorname{Tor}^{S}_i(S/I,k)_\bullet
		& \xrightarrow{\ \overline{\phi}_\ast\ } &
		\operatorname{Tor}^{S}_i(S/I',k)_\bullet\\[2pt]
		\big\downarrow\scriptstyle{\delta_i}^{\cong} & & \big\downarrow\scriptstyle{\delta_i}^{\cong}\\[2pt]
		\operatorname{Tor}^{S}_{i-1}(I,k)_\bullet
		& \xrightarrow{\ \phi_\ast\ } &
		\operatorname{Tor}^{S}_{i-1}(I',k)_\bullet
	\end{array}
	\]
	commutes. Taking dimensions yields the stated compatibility for Betti numbers. This proves \textup{(b)}.
	
	For \textup{(c)}, apply \textup{(b)} to each structure map \(\phi_{s\to t}:I_s\to I_t\) in the family
	and the induced \(\overline{\phi}_{s\to t}:S/I_s\to S/I_t\). The preceding commutative squares
	show that the identifications from \textup{(a)} define an isomorphism of persistence modules
	\[
	\bigl(\operatorname{Tor}^{S}_i(S/I_t,k)_\bullet\bigr)_{t\in T}
	\ \cong\
	\bigl(\operatorname{Tor}^{S}_{i-1}(I_t,k)_\bullet\bigr)_{t\in T}
	\qquad (i\ge 1),
	\]
	i.e., they intertwine all structure maps. The degree–\(0\) case for \(i=0\) is as in \textup{(a)}.
\end{proof}

\noindent
Finally, if each level $I_t$ admits a Betti splitting compatible with the inclusions
across $t$, then the persistent Betti numbers split additively. In the degenerate case
$a=b$ this specializes to the usual Betti splitting formula.

\begin{proposition}[Persistent Betti splitting]\label{prop:persistent-betti-splitting}
	Let $\{I_t\}_{t\in T}$ be a filtration of monomial ideals in $S=k[x_1,\dots,x_n]$ with
	$T$ totally ordered and $a\le b$.
	Assume that for each $t$ we have a decomposition
	\[
	I_t \;=\; J_t + K_t
	\]
	such that
		$G(I_t)=G(J_t)\,\dot\cup\,G(K_t)$ (disjoint union of minimal generators) and
		\[
		\beta_{i,j}(I_t) \;=\; \beta_{i,j}(J_t)+\beta_{i,j}(K_t)+\beta_{i-1,j}(J_t\cap K_t)
		\quad \text{for all } i,j.
		\]
		Equivalently, in the short exact sequence
		\[
		0\to J_t\cap K_t \longrightarrow J_t\oplus K_t \longrightarrow I_t \to 0,
		\]
		the induced maps
		$\mathrm{Tor}_i^S(k,J_t\cap K_t)_j \to \mathrm{Tor}_i^S(k,J_t)_j$ and
		$\mathrm{Tor}_i^S(k,J_t\cap K_t)_j \to \mathrm{Tor}_i^S(k,K_t)_j$ are $0$ for all $i,j$.
		
	Then the persistent graded Betti numbers satisfy:
	\[
	\beta^{\,a, b}_{i,j}(I_\bullet)
	\;\geq \;
	\beta^{\,a, b}_{i,j}(J_\bullet)
	\;+\;
	\beta^{\,a, b}_{i,j}(K_\bullet)
	\;+\;
	\beta^{\,a, b}_{i-1,j}\bigl((J\cap K)_\bullet\bigr)
	\qquad\text{for all } i,j,
	\]
	and likewise in the multigraded form with $j$ replaced by $W\subseteq[n]$.
\end{proposition}

\begin{proof}[Sketch]
	Fix $i,j$. For each $t$ the short exact sequence
	\[
	0\to J_t\cap K_t \longrightarrow J_t\oplus K_t \longrightarrow I_t \to 0
	\]
	induces the long exact sequence in $\mathrm{Tor}^S(-,k)$; by (1) the maps
	$\mathrm{Tor}_i(k,J_t\cap K_t)_j\to \mathrm{Tor}_i(k,J_t)_j$ and
	$\mathrm{Tor}_i(k,J_t\cap K_t)_j\to \mathrm{Tor}_i(k,K_t)_j$ are zero, so the long exact
	sequence breaks degreewise into short exact sequences
	\begin{equation}\label{eq:ses}
		0\;\longrightarrow\;
		X_t \xrightarrow{\ \iota_t\ } Y_t \xrightarrow{\ \partial_t\ } Z_t \;\longrightarrow\; 0,
	\end{equation}
	where
	\[
	X_t:=\mathrm{Tor}_i(k,J_t)_j\oplus \mathrm{Tor}_i(k,K_t)_j,\quad
	Y_t:=\mathrm{Tor}_i(k,I_t)_j,\quad
	Z_t:=\mathrm{Tor}_{i-1}(k,J_t\cap K_t)_j.
	\]
	By (2) and naturality of $\mathrm{Tor}$, for $a\le b$ these assemble into a commutative diagram
	of short exact sequences:
	\[
	\begin{tikzcd}
		0 \ar[r] & X_a \ar[r, "\iota_a"] \ar[d, "x"] &
		Y_a \ar[r, "\partial_a"] \ar[d, "y"] &
		Z_a \ar[r] \ar[d, "z"] & 0\\
		0 \ar[r] & X_b \ar[r, "\iota_b"] &
		Y_b \ar[r, "\partial_b"] &
		Z_b \ar[r] & 0
	\end{tikzcd}
	\]
	(all maps $k$–linear).
	
	Now apply the elementary linear–algebra lemma for maps between short exact sequences of
	vector spaces: in any commutative diagram as above,
	\[
	\mathrm{rank}(y)\ \ge\ \mathrm{rank}(x)\ +\ \mathrm{rank}(z),
	\]
	with equality if and only if
	\[
	\operatorname{im}(y)\cap \operatorname{im}(\iota_b)\;=\;\iota_b\bigl(\operatorname{im}(x)\bigr)
	\quad\text{(equivalently, $0\to \operatorname{im}x\to \operatorname{im}y\to \operatorname{im}z\to 0$ is exact).}
	\]
	By definition of persistent Betti ranks,
	\[
	\beta^{\,a, b}_{i,j}(I_\bullet)=\mathrm{rank}(y),\quad
	\beta^{\,a, b}_{i,j}(J_\bullet)=\mathrm{rank}\!\bigl(\mathrm{Tor}_i(k,J_a)_j\to \mathrm{Tor}_i(k,J_b)_j\bigr),
	\]
	\[
	\beta^{\,a, b}_{i,j}(K_\bullet)=\mathrm{rank}\!\bigl(\mathrm{Tor}_i(k,K_a)_j\to \mathrm{Tor}_i(k,K_b)_j\bigr),\quad
	\beta^{\,a, b}_{i-1,j}\bigl((J\cap K)_\bullet\bigr)=\mathrm{rank}(z).
	\]
	Since $\mathrm{rank}(x)$ is the sum of the $J$ and $K$ ranks (on a direct sum), the lemma yields
	\[
	\beta^{\,a, b}_{i,j}(I_\bullet)
	\ \ge\
	\beta^{\,a, b}_{i,j}(J_\bullet)
	\ +\
	\beta^{\,a, b}_{i,j}(K_\bullet)
	\ +\
	\beta^{\,a, b}_{i-1,j}\bigl((J\cap K)_\bullet\bigr).
	\]
	The multigraded version follows by applying the same argument in each multidegree.

	If, in addition, the middle images are exact (e.g. when the short exact sequences
	\eqref{eq:ses} split functorially across $t$ via a mapping-cone construction compatible with the
	filtration), then equality holds in the display above.
\end{proof}

All results in this section are stated for {general} filtrations of graded monomial ideals
$(I_t)_{t\in T}$ and of simplicial complexes $(\Delta_t)_{t\in T}$. In particular, the persistent
Hochster formula and the functorial Betti shift apply degreewise to any such filtration, and
specialize to the usual (non-persistent) identities when $a=b$.

In applications below we restrict to filtrations coming from graphs. There are two standard
encodings:
\begin{itemize}
	\item {Edge ideals:} for a filtration of graphs $(G_t)_{t\in T}$ on a fixed vertex set $[n]$,
	take $I_t:=I(G_t)\subset S=k[x_1,\dots,x_n]$. Persistent Betti numbers
	$\beta^{\,a\to b}_{i,\bullet}(I_\bullet)$ then measure the ranks of the maps on
	$\operatorname{Tor}$ induced by $I(G_a)\hookrightarrow I(G_b)$.
	\item {Stanley–Reisner rings of independence complexes:} set $\Delta_t:=\operatorname{Ind}(G_t)$
	and $k[\Delta_t]=S/I_{\Delta_t}$. The persistent Hochster formula identifies
	$\beta^{\,a, b}_{i,W}\!\bigl(k[\Delta_\bullet]\bigr)$ with the rank of the cohomology map
	$\widetilde H^{\,|W|-i-1}\!\bigl(\operatorname{Ind}(G_b[W]);k\bigr)\to
	\widetilde H^{\,|W|-i-1}\!\bigl(\operatorname{Ind}(G_a[W]);k\bigr)$,
	where $G_t[W]$ denotes the induced subgraph on $W$.
\end{itemize}
Thus every persistent statement above immediately yields its graph-theoretic counterpart;
when $a=b$ we recover the classical formulas for $I(G)$ and for $k[\operatorname{Ind}(G)]$.

\section{Persistent minimal primes for edge ideals}

Let $S=k[x_1,\dots,x_n]$ be the standard graded polynomial ring and let
$\{G_t\}_{t\in\mathbb{R}}$ be an increasing filtration of simple graphs on the fixed
vertex set $V=\{1,\dots,n\}$, i.e.
\[
E_s \subseteq E_t \qquad \text{whenever } s\le t .
\]
For each $t$, define the {edge ideal}
\[
I_t \;:=\; I(G_t) \;=\; \big\langle\, x_ix_j \ \bigm|\ \{i,j\}\in E_t \,\big\rangle \;\subseteq\; S .
\]
Because the graph filtration is increasing, the edge ideals form an increasing chain
\[
I_s \;\subseteq\; I_t \qquad \text{for all } s\le t .
\]

It is classical that the minimal primes of $I(G)$ are in bijection with the minimal
vertex covers of $G$: if $C\subseteq V$ is a minimal vertex cover of $G$, then
\[
P_C \;:=\; (x_i \mid i\in C) \;\subseteq\; S
\]
is a minimal prime of $I(G)$, and every minimal prime has this form.  For each $t$
we therefore set
\[
\mathcal{P}^t \;:=\; \operatorname{MinAss}(S/I_t)
\;=\;
\bigl\{\, P_C \ \bigm|\ C \text{ a minimal vertex cover of } G_t \,\bigr\}.
\]
To record the size profile of covers, we stratify by cardinality:
\[
\mathcal{P}_k^t \;:=\; \bigl\{\, P_C\in \mathcal{P}^t \ \bigm|\ |C|=k \,\bigr\},
\qquad
\mathcal{P}^t \;=\; \bigsqcup_{k\ge 0}\, \mathcal{P}_k^t .
\]

Given $t<t'$, we say that a prime $P_C\in\mathcal{P}^t$ {persists} to level
$t'$ if it is still minimal over $I_{t'}$, i.e.\ $P_C\in\mathcal{P}^{t'}$.
Equivalently, the vertex cover $C$ is minimal for both $G_t$ and $G_{t'}$.
We collect the persistent components (optionally by size) as
\[
\mathcal{P}^{t,t'} \;:=\; \mathcal{P}^t \cap \mathcal{P}^{t'},
\qquad
\mathcal{P}^{t,t'}_k \;:=\; \mathcal{P}^t_k \cap \mathcal{P}^{t'}_k ,
\]
and define the {persistent prime numbers}
\[
\Pi^{t,t'} \;:=\; \bigl|\mathcal{P}^{t,t'}\bigr|,
\qquad
\Pi^{t,t'}_k \;:=\; \bigl|\mathcal{P}^{t,t'}_k\bigr| .
\]
A (vertex-cover) prime $P_C$ has {birth time}
\[
b(P_C) \;:=\; \inf\{\, t \mid P_C \in \mathcal{P}^t \,\}
\]
(with $b(P_C)=+\infty$ if the set is empty), and {death time}
\[
d(P_C) \;:=\; \inf\{\, t> b(P_C) \mid P_C \notin \mathcal{P}^t \,\}.
\]
Thus $P_C$ is present (i.e.\ minimal over $I_t$) exactly for $t\in [\,b(P_C),\,d(P_C))$.
Equivalently, the cover $C$ is minimal for $G_t$ precisely on this interval.

As $t$ increases, edges are added to $G_t$, so $I_t$ gains generators.  Consequently,
a vertex cover that is minimal at time $t$ may {die} at a later time $t'$ by
either ceasing to cover a newly added edge or remaining a cover but no longer being
minimal.  Conversely, new minimal covers (hence new minimal primes) may {be born}
at later times as the combinatorics tighten.  The invariants
$\{\Pi_k^{t,t'}\}_{k,t,t'}$ (and their unstratified version $\Pi^{t,t'}$) provide a
purely algebraic barcode for the persistence of minimal primes of the edge-ideal
filtration, paralleling persistent homology but expressed through minimal vertex covers.

\section{Hypergraphs, edge ideals, and facet ideals}\label{sec:hypergraphs}

The edge ideal construction extends naturally from graphs to general hypergraphs \cite{ha2008monomial}.
A (finite) hypergraph on the vertex set $V=\{1,\dots,n\}$ is a pair $H=(V,\mathcal{E})$,
where $\mathcal{E}\subseteq 2^V$ is a collection of nonempty subsets (the \emph{hyperedges}).
We say that $H$ is \emph{$d$--uniform} if every hyperedge has cardinality $d$.

Given a hypergraph $H=(V,\mathcal{E})$ and the polynomial ring $S=k[x_1,\dots,x_n]$, its
\emph{edge ideal} (or hyperedge ideal) is the squarefree monomial ideal
\[
I(H)\;:=\;\bigl( x_F \mid F\in\mathcal{E}\bigr)\ \subseteq\ S,
\qquad
x_F\;:=\;\prod_{i\in F}x_i.
\]
As in the graph case, $I(H)$ is squarefree and hence radical. Vertex covers and minimal
vertex covers are defined exactly as before: a subset $C\subseteq V$ is a vertex cover (or
transversal) of $H$ if it meets every hyperedge, and it is minimal if it contains no
smaller vertex cover. The same argument used for graphs shows that the minimal primes of
$I(H)$ are precisely the monomial primes
\[
\mathfrak p_C \;=\;\langle x_i : i\in C\rangle
\]
with $C$ a minimal vertex cover of $H$.

It is again convenient to pass to an independence complex. The \emph{independence complex}
of $H$ is the simplicial complex
\[
\mathrm{Ind}(H)\;:=\;\bigl\{\,\sigma\subseteq V \ \bigm|\ \text{$\sigma$ contains no hyperedge of $H$}\,\bigr\}.
\]
As in the graph case, the edge ideal of $H$ coincides with the Stanley--Reisner ideal of
$\mathrm{Ind}(H)$:
\[
I(H)\;=\;I_{\mathrm{Ind}(H)}.
\]
Equivalently, the minimal nonfaces of $\mathrm{Ind}(H)$ are exactly the hyperedges of $H$, and the
minimal monomial generators of $I_{\mathrm{Ind}(H)}$ are the monomials $x_F$ for $F\in\mathcal{E}$.

Consequently, all statements formulated above for edge ideals of graphs extend verbatim to
hypergraph edge ideals upon replacing $G$ by $H$ and $\mathrm{Ind}(G)$ by $\mathrm{Ind}(H)$. In particular,
Hochster's formula and its persistent version apply to $S/I(H)$ via the complexes
$\mathrm{Ind}(H)$, and the correspondence between minimal primes and minimal vertex covers carries
over unchanged.

More generally, all constructions in this paper are functorial in the underlying graded
monomial ideals and have been stated for arbitrary filtrations of such ideals and their
Stanley--Reisner rings. Thus if $\{H_t\}_{t\in T}$ is a filtration of hypergraphs on a fixed
vertex set and we set $I_t:=I(H_t)\subseteq S$, then:
\begin{itemize}
	\item The persistent graded Betti numbers $\beta^{a,b}_{i,\bullet}(I_\bullet)$ and
	$\beta^{a,b}_{i,\bullet}(S/I_\bullet)$ are defined exactly as in the graph case;
	\item The persistent Hochster formula identifies these Betti ranks with the ranks of
	the induced maps on (co)homology of the induced subcomplexes
	$\mathrm{Ind}(H_t)_W\subseteq \mathrm{Ind}(H_{t'})_W$;
	\item The functorial Betti shift between $I_t$ and $S/I_t$ and the persistent Betti
	splitting inequalities hold unchanged;
	\item The theory of persistent minimal primes specializes to persistent minimal
	vertex covers (minimal transversals) of the hypergraph filtration.
\end{itemize}
In particular, every result stated for filtrations of graphs extends directly to filtrations
of hypergraphs, and, more broadly, to any filtration of squarefree monomial ideals, including
facet ideals of simplicial complexes.

\section{Applications}

\subsection{Application: Genome Classification via Persistent Edge Ideals}

Genomic classification plays a central role in understanding evolutionary relationships, pathogen diversity, and functional similarities among organisms. With the rapid accumulation of viral and genomic sequences in public databases, there is a growing need for reliable  alignment-free approaches that capture intrinsic structural patterns within DNA sequences. Inspired by recent developments in algebraic topology and persistent homology\cite{hozumi2024revealing}, we propose a method based on {persistent edge ideals}, which encodes the positional and combinatorial structure of $k$-mer occurrences along genomic sequences.

We evaluate our approach on a curated mitochondrial genome dataset as curated in \cite{MillanArias2022DeLUCS}. The collection comprises {545} complete mitochondrial sequences
distributed across seven families with near-uniform cluster sizes between {70} and {80}
sequences each: Cyprinidae (80), Cobitidae (80), Balitoridae (75), Nemacheilidae (80),
Xenocyprididae (80), Acheilognathidae (70), and Gobionidae (80). Genome lengths are tightly
concentrated (minimum {16{,}061}\,bp, average {16{,}610}\,bp, maximum
{17{,}282}\,bp), providing a balanced benchmark at this taxonomic depth (see the table
entry for Dataset~5). Our comparative setup follows Hozumi and Wei~\cite{hozumi2024revealing}.

Let $\mathcal{A}$ denote the nucleotide alphabet and fix $k = 4$. For a DNA sequence $S = s_1 s_2 \cdots s_N \in \mathcal{A}^N$, we consider its collection of $k$-mers $\boldsymbol{x} \in \mathcal{A}^k$ and record the set of positions
\[
S^{\boldsymbol{x}} = \bigl\{\, i \in [1, N - k + 1] \ \big|\ s_i s_{i+1}\cdots s_{i+k-1} = \boldsymbol{x} \,\bigr\}.
\]
We associate to each $k$-mer $\boldsymbol{x}$ a graph $G^{\boldsymbol{x}}$ whose vertices correspond to occurrences in $S^{\boldsymbol{x}}$. Two vertices $i,j \in S^{\boldsymbol{x}}$ are connected whenever their positional distance satisfies
\[
|i - j| \leq r,
\]
where $r$ is a prescribed threshold controlling local connectivity. This construction produces a distance-based interaction graph that captures the spatial distribution of identical $k$-mers along the sequence. To probe structural changes across different interaction radii, we consider a filtration
\[
G^{\boldsymbol{x}}_{r_0} \subseteq G^{\boldsymbol{x}}_{r_1} \subseteq \cdots \subseteq G^{\boldsymbol{x}}_{r_T},
\]
and for each graph in the filtration define its {edge ideal}
\[
I(G^{\boldsymbol{x}}_r) = \big\langle\, x_i x_j \;\big|\; \{i,j\} \in E(G^{\boldsymbol{x}}_r) \,\big\rangle.
\]
The complement graph $\overline{G^{\boldsymbol{x}}_r}$ is also of particular interest: while $G^{\boldsymbol{x}}_r$ is typically chordal, its complement $\overline{G^{\boldsymbol{x}}_r}$ is generally non-chordal and admits a linear resolution. This property allows one to track persistent algebraic invariants of $\overline{G^{\boldsymbol{x}}_r}$ along the filtration. For each filtration step $r$, we compute the graded Betti numbers of the edge ideal $I(\overline{G^{\boldsymbol{x}}_r})$. Among these, the coefficient $\beta_{n-2,n}$---where $n = |S^{\boldsymbol{x}}|$ denotes the number of occurrences of $\boldsymbol{x}$---captures key algebraic information about the complement graph's linear resolution. The resulting curve
\[
r \longmapsto \beta_{n-2,n}\!\left(I(\overline{G^{\boldsymbol{x}}_r})\right)
\]
is referred to as the {persistent Betti curve} associated with $\boldsymbol{x}$. Concatenating these curves across all $k$-mers yields a comprehensive feature representation of the genome:
\[
\boldsymbol{v}^k_S = \bigl( \beta_{n-2,n}^{\boldsymbol{x}}(r_t) \bigr)_{\boldsymbol{x} \in \mathcal{A}^k,\, t=1,\dots,T}.
\]
Pairwise distances between sequences are then computed as Euclidean distances between their concatenated feature vectors, providing an alignment-free measure of genomic similarity.

Following the nearest-neighbor evaluation protocol introduced by Sun~{et~al.}~\cite{sun2021geometric}, a test sequence is considered correctly classified if its nearest neighbor (under the above feature-space distance) belongs to the same viral family. Using $k=4$ and unweighted distances, the proposed persistent edge ideal representation achieved the following performance:

\begin{table}[h!]
	\centering
	\begin{tabular}{l c c c c c}
		\hline
		& \textbf{Accuracy} & $\mathbf{F_1}$ & \textbf{Balanced accuracy} & \textbf{Recall} & \textbf{Precision} \\
		\hline
		\textbf{Score} & 0.8385 & 0.8378 & 0.8401 & 0.8401 & 0.8415 \\
		\hline
	\end{tabular}
	\caption{Classification performance summary.}
	\label{tab:performance-summary}
\end{table}

These results demonstrate that persistent edge ideal features capture informative structural patterns within viral genomes.

\subsection{Molecular Representation}

\begin{figure}[!t]
	\centering
	\begin{subfigure}[t]{0.20\textwidth}
		\centering
		\includegraphics[width=\textwidth]{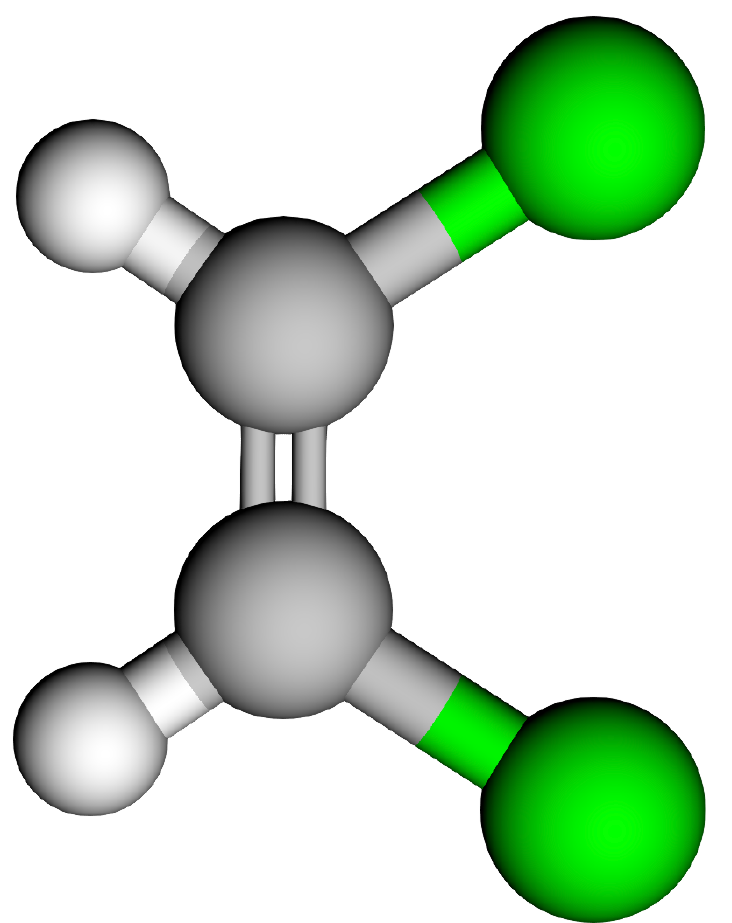}
		\caption{}
		\label{fig:cisisomer}
	\end{subfigure}
	\hspace{3cm}
	\begin{subfigure}[t]{0.24\textwidth}
		\centering
		\includegraphics[width=\textwidth]{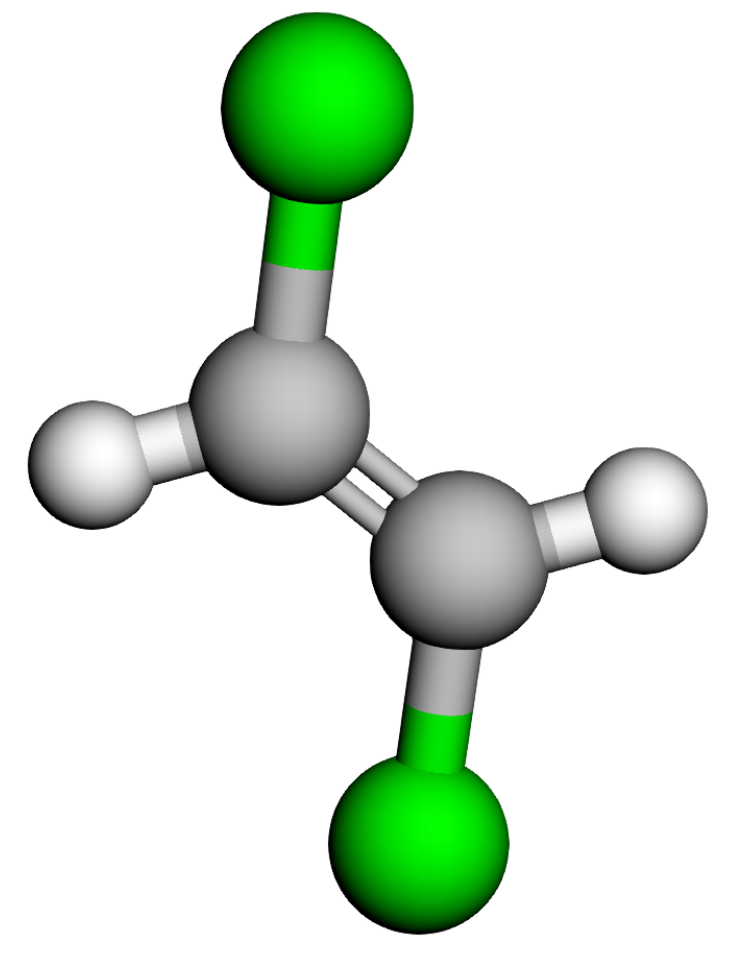}
		\caption{}
		\label{fig:transisomer}
	\end{subfigure}
	\caption{Visualization of (a) the cis- isomer \cite{PubChem_CID643833} and (b) the trans- isomer \cite{PubChem_CID638186}.}
	\label{fig:isomers}
\end{figure}

Recent advances in artificial intelligence have driven the development of algebraic and topological frameworks capable of encoding molecular or network identity through connectivity and atomic arrangement. Graph-based representations naturally capture these structural relationships, with vertices denoting atoms and edges representing chemical bonds. Methods inspired by {topological data analysis} and {persistent homology} have therefore been widely proposed in molecular and biological current research directions~\cite{xia2019persistent, chen2023path, anand2022topological}.

From an algebraic standpoint, molecular graphs can be characterized by their {edge ideals}, which translate adjacency information into monomial generators. These ideals form algebraic descriptors that encode the combinatorial and geometric structure of the underlying graph. In this work, we propose the use of {persistent edge ideals} by incorporating a filtration parameter that controls the inclusion of edges based on a physical or geometric criterion, such as interatomic distance.

Molecular function is fundamentally linked to the three-dimensional geometry of atoms within a molecule. Even when two compounds share the same molecular formula, distinct spatial arrangements of their atoms can lead to differences in polarity, reactivity, and stability. This phenomenon, known as {isomerism}, arises from variations either in bonding connectivity or in the relative spatial positioning of atoms. While conventional analyses distinguish isomers using experimental measurements of physical or spectroscopic properties, modern computational methods allow such distinctions to be characterized directly from geometric and topological information~\cite{SuwayyidWei2024_PersistentMayerDirac, wee2023persistent}.

To demonstrate this capability, we apply the {persistent edge ideal} framework to two geometric isomers of dichloroethene, {cis-} and {trans-}\(\mathrm{C_2H_2Cl_2}\) depicted in Figure \ref{fig:isomers}. Although both share identical connectivity, the orientation of the chlorine atoms relative to the carbon–carbon double bond differs: in the {cis}-form, the chlorine atoms lie on the same side of the molecular plane, whereas in the {trans}-form, they occupy opposite sides. Using the atomic coordinates of each isomer, a Vietoris–Rips filtration is constructed over the range \([0,5]\,\text{\AA}\), and the graded Betti numbers of the resulting edge ideals are computed across dimensions.

The filtration reveals distinct topological behavior between the two forms. The {cis}-isomer, being more compact, exhibits earlier activations of edge-related Betti generators, while the more elongated {trans}-isomer produces delayed activated edges demonstrated in the graded Betti numbers \(\beta_{i,i+2}\). Moreover, due to the compactness of the cis-isomer, some combinatorial features live shorter than those in the trans-isomer revealed in the graded Betti numbers \(\beta_{i,i+3}\). These differences indicate that persistent edge ideals are sensitive to subtle geometric variations, providing a purely algebraic–topological signature capable of distinguishing between molecular isomers. Figures~\ref{fig:cis}--\ref{fig:trans} display the corresponding atomic geometries and persistence diagrams. Clearly, these isomers have very different characteristics of persistent graded Betti curves.

\begin{figure}[!t]
	\centering
	\begin{subfigure}[t]{0.48\textwidth}
		\centering
		\includegraphics[width=\textwidth]{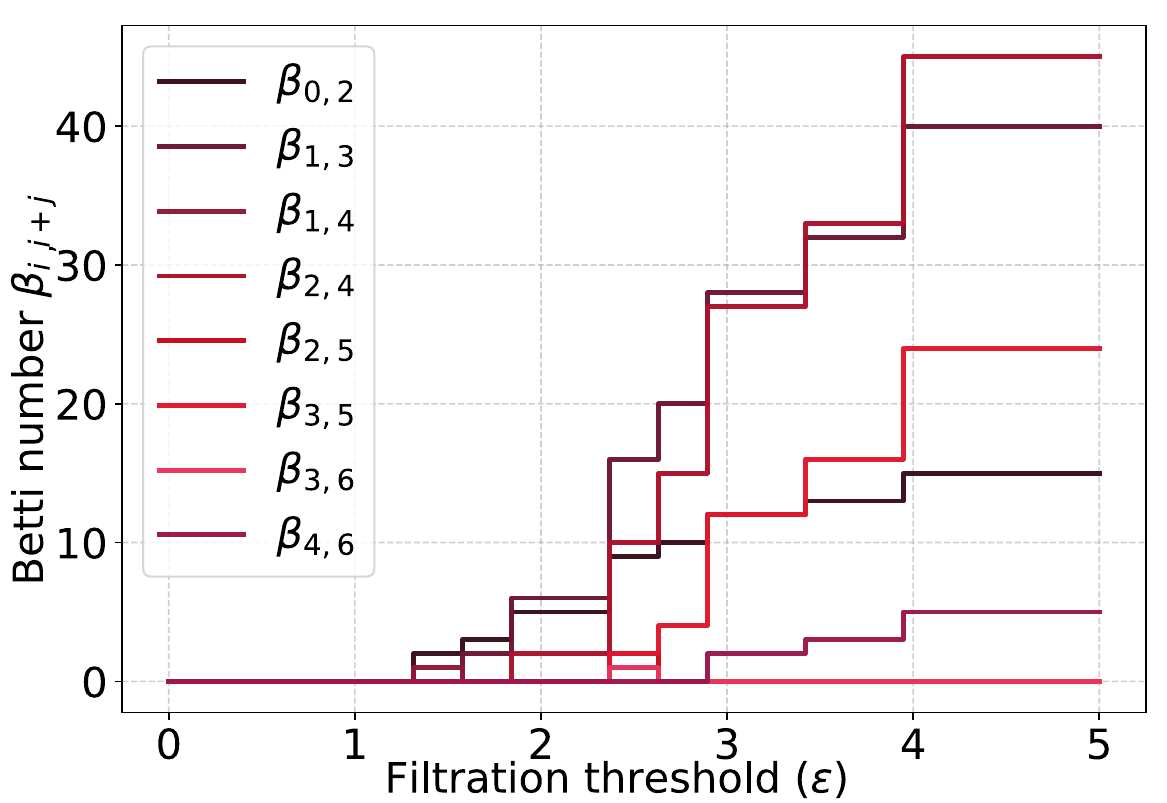}
		\caption{}
		\label{fig:cis}
	\end{subfigure}
%	\hfill
	\begin{subfigure}[t]{0.48\textwidth}
		\centering
		\includegraphics[width=\textwidth]{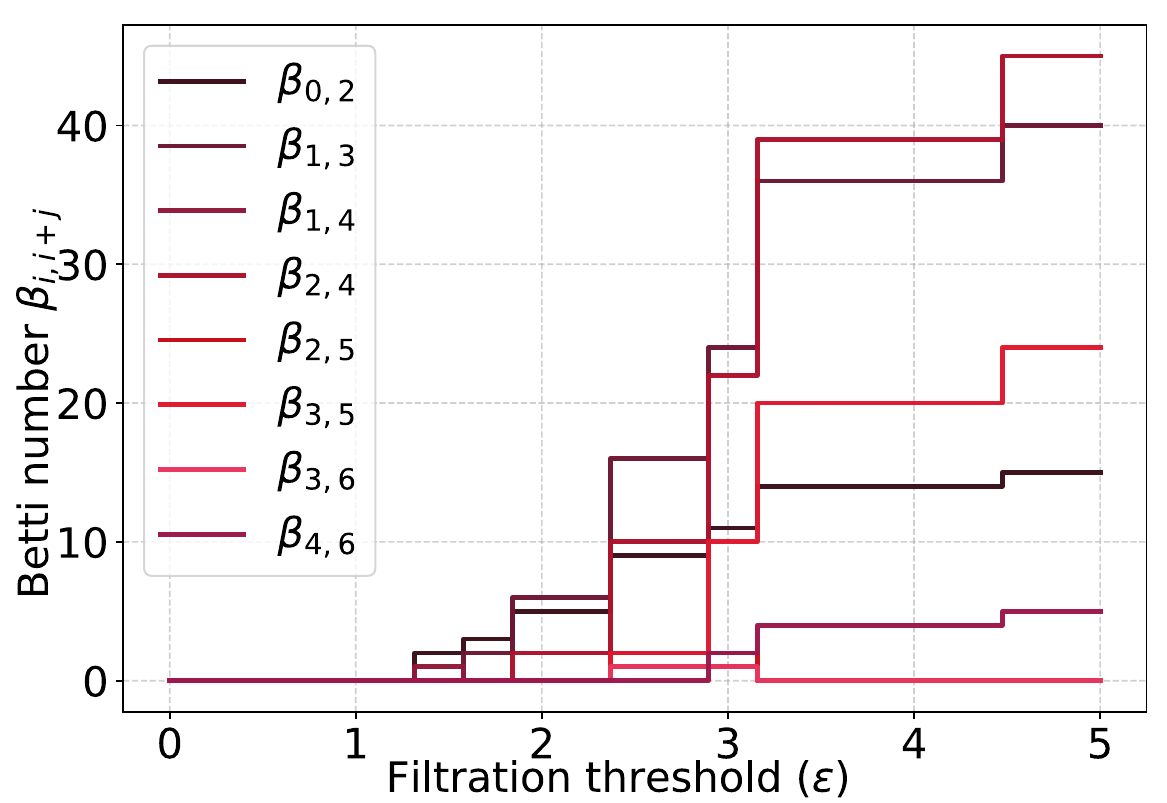}
		\caption{}
		\label{fig:trans}
	\end{subfigure}
	\caption{Comparison between (a) the cis- isomer and (b) the trans- isomer graded persistent Betti curves of edge ideals of the dichloroethene molecule derived from the Vietoris--Rips edge ideal filtration.}
	\label{fig:betti-comparison}
\end{figure}

\section{Discussion}

The present framework extends the ideas introduced in the Persistent Stanley--Reisner Theory (PSRT)\cite{SuwayyidWei2026_PSRT}, but with notable conceptual and technical differences.  
In PSRT, persistence was defined through the {change of the combinatorial decomposition}—for instance, through the evolution of facets or Stanley--Reisner generators—and the persistence module was constructed directly on those decompositions, with a subsequent analysis of stability.  
In contrast, the current approach begins from the functorial and homological viewpoint: persistence is defined intrinsically through the $\operatorname{Tor}$ functor, with structure maps induced naturally by inclusions in the filtration.  Stability and naturality thus arise as direct consequences of the functoriality of $\operatorname{Tor}$ and the compatibility of Hochster’s correspondence with these maps.

Furthermore, while PSRT formulated the persistent Hochster formula by identifying the corresponding persistent (co)homology groups, here we start with the $\operatorname{Tor}$-based definition of persistent Betti numbers and then derive the Hochster description as a corollary.  Hence, the Hochster formula appears not as a defining statement but as a natural compatibility result between algebraic and topological realizations.  
Another distinction lies in the treatment of classical relations such as the Betti shift and Betti splitting.
The results presented here are not restricted to graphs or simplicial complexes but remain valid for any filtration of monomial ideals, including those arising from hypergraphs, hyperedge ideals, and facet ideals.  In this sense, the framework applies to a broad class of combinatorial and algebraic objects.

In parallel to our approach, Hu et al. presented {persistent ideals} as the associated primes of Stanley–Reisner ideals or of edge ideals taken along a filtration and used their births and deaths as commutative-algebraic barcodes~\cite{Hu2025CommAlgTDA}. In the present work, we concentrate on the {minimal} associated primes, which correspond directly to minimal vertex covers (for edge ideals) and thus more transparently reflect and allow finer control over the combinatorial structure of the underlying monomial ideals.

\section{Conclusion}
This work develops a functorial and homological framework for persistence in commutative algebra, extending the combinatorial and simplicial complex-based approach of the Persistent Stanley--Reisner Theory (PSRT) to a derived-functor setting. By defining persistent graded Betti numbers through the $\operatorname{Tor}$ functor, we established a functorial notion of persistence that naturally inherits exactness, naturality, and duality properties from homological algebra. This perspective unifies several classical identities---Hochster’s formula, the Betti shift, and Betti splitting---into a single persistent framework that holds across filtrations of graded monomial ideals.

The persistent Hochster formula derived herein provides an explicit bridge between algebraic persistence and topological persistence. It identifies the persistent graded Betti numbers of a filtration of Stanley--Reisner rings with the ranks of induced maps in the cohomology of induced subcomplexes. Consequently, this establishes a direct correspondence between algebraic and topological persistence modules, recovering classical persistent homology as a special case. When applied to graphs, hypergraphs,  and their edge ideals, this formalism yields a persistent Betti theory that tracks the evolution of syzygies, minimal generators, and higher-order algebraic relations under edge or vertex filtration. 

In addition to the functorial Betti–based viewpoint, we introduced a prime–level summary for edge–ideal filtration by tracking {persistent minimal primes}. Recording the births and deaths of these primes across $t$ yields the collections $\mathcal{P}^{t,t'}$ and their cardinalities $\Pi^{t,t'}$ which function as an algebraic barcode complementary to persistent graded Betti data. These prime barcodes provide interpretable indicators of structural change—identifying exactly which vertex covers remain minimal as edges are added—while persistent Betti numbers quantify the evolution of generators and syzygies. Together, these invariants deliver a multi–resolution description of filtered combinatorial structure that is both computationally tractable and directly tied to graph–theoretic semantics, enhancing downstream tasks.

Beyond theoretical contributions, our framework is directly applicable in several data-driven settings. We use persistent edge ideals to build $k$-mer interaction graphs along genomes and track graded Betti numbers across radii, yielding fixed-length {persistent graded Betti vectors} on a curated mitochondrial benchmark~\cite{MillanArias2022DeLUCS}. This representation supports alignment-free classification under a nearest-neighbor protocol, capturing family-level structure in an interpretable algebraic form. For molecular graphs with 3D coordinates, filtrations by interatomic distance produce persistent edge ideals whose graded Betti profiles discriminate geometric isomers (e.g., cis/trans dichloroethene); the cis configuration exhibits earlier but shorter-lived linear and near-linear strand features, while the trans configuration yields delayed yet longer-lived ones, reflecting compact versus elongated geometries in a purely algebraic signal. These example applications demonstrate the potential of the proposed method in data science and machine learning.

\section*{Acknowledgments}
The work of Suwayyid was supported by the King Fahd University of Petroleum and Minerals. The work of Wei was supported by NIH grant  R35GM148196 and    MSU research Foundation.

\bibliographystyle{unsrt}
\bibliography{references}

@article{ren2025interpretability,
  title={Interpretability and representability of commutative algebra, algebraic topology, and topological spectral theory for real-world data},
  author={Ren, Yiming and Wei, Guowei},
  journal={Advanced Intelligent Discovery},
  year={2025}
}

@article{ha2008monomial,
  title={Monomial ideals, edge ideals of hypergraphs, and their graded Betti numbers},
  author={H{\`a}, Huy T{\`a}i and Van Tuyl, Adam},
  journal={Journal of Algebraic Combinatorics},
  volume={27},
  number={2},
  pages={215--245},
  year={2008},
  publisher={Springer}
}

@article{hochster1969prime,
  title={Prime ideal structure in commutative rings},
  author={Hochster, Melvin},
  journal={Transactions of the American Mathematical Society},
  volume={142},
  pages={43--60},
  year={1969}
}

@article{su2025topological,
  title={Topological Data Analysis and Topological Deep Learning Beyond Persistent Homology-A Review},
  author={Su, Zhe and Liu, Xiang and Hamdan, Layal Bou and Maroulas, Vasileios and Wu, Jie and Carlsson, Gunnar and Wei, Guo-Wei},
  journal={Artificial Intelligence Review},
   pages={DOI: 10.1007/s10462-025-11462-w},
  year={2025}
}

@article{wee2025review,
	title        = {A Review of Topological Data Analysis and Topological Deep Learning in Molecular Sciences},
	author       = {Wee, JunJie and Jiang, Jian},
	journal      = {Journal of Chemical Information and Modeling},
	year         = {2025},
	volume       = {65},
	number       = {23},
	pages        = {12691--12706},
	month        = dec,
	doi          = {10.1021/acs.jcim.5c02266},
	pmid         = {41235667},
	publisher    = {American Chemical Society},
	note         = {Epub 2025-11-14}
}

@article{wei2025persistent,
  title={Persistent topological laplacians---a survey},
  author={Wei, Xiaoqi and Wei, Guo-Wei},
  journal={Mathematics},
  volume={13},
  number={2},
  pages={208},
  year={2025},
  publisher={MDPI}
}

@article{zia2025gbnl,
  title={GBNL: Graded Betti Number Learning of Complex Biological Data},
  author={Zia, Mushal and Suwayyid, Faisal and Wei, Guo-Wei},
  journal={arXiv preprint arXiv:2510.23187},
  year={2025}
}

@article{sun2021geometric,
	author  = {Nan Sun and Shaojun Pei and Lily He and Changchuan Yin and Rong Lucy He and Stephen S.-T. Yau},
	title   = {Geometric construction of viral genome space and its applications},
	journal = {Computational and Structural Biotechnology Journal},
	volume  = {19},
	pages   = {4226--4234},
	year    = {2021},
	doi     = {10.1016/j.csbj.2021.07.028},
	url     = {https://doi.org/10.1016/j.csbj.2021.07.028}
}

@article{hozumi2024revealing,
	author  = {Yuta Hozumi and Guo-Wei Wei},
	title   = {Revealing the Shape of Genome Space via K-mer Topology},
	journal = {SIAM Journal on Life Sciences },
	year    = {2026},
	url     = {https://arxiv.org/abs/2412.20202}
}

@article{carlsson2009topology,
	title        = {Topology and data},
	author       = {Carlsson, Gunnar},
	journal      = {Bulletin of the American Mathematical Society},
	volume       = {46},
	number       = {2},
	pages        = {255--308},
	year         = {2009},
	publisher    = {American Mathematical Society},
	doi          = {10.1090/S0273-0979-09-01249-X}
}

@incollection{edelsbrunner2008persistent,
	title        = {Persistent homology---a survey},
	author       = {Edelsbrunner, Herbert and Harer, John},
	booktitle    = {Surveys on Discrete and Computational Geometry: Twenty Years Later},
	volume       = {453},
	pages        = {257--282},
	year         = {2008},
	publisher    = {American Mathematical Society},
	doi          = {10.1090/conm/453/08802}
}

@article{zomorodian2005computing,
	title        = {Computing persistent homology},
	author       = {Zomorodian, Afra and Carlsson, Gunnar},
	journal      = {Discrete \& Computational Geometry},
	volume       = {33},
	number       = {2},
	pages        = {249--274},
	year         = {2005},
	publisher    = {Springer},
	doi          = {10.1007/s00454-004-1146-y}
}

@article{xia2019persistent,
	title        = {Persistent homology analysis of osmolyte molecular aggregation and their hydrogen‐bonding networks},
	author       = {Xia, Kelin and Anand, D. Vijay and Shikhar, Saxena and Mu, Yuguang},
	journal      = {Physical Chemistry Chemical Physics},
	volume       = {21},
	number       = {37},
	pages        = {21038--21048},
	year         = {2019},
	doi          = {10.1039/C9CP03009C},
	issn         = {1463-9076},
	url          = {https://doi.org/10.1039/C9CP03009C}
}

@article{chen2023path,
	title        = {Path topology in molecular and materials sciences},
	author       = {Chen, Dong and Liu, Jian and Wu, Jie and Wei, Guo-Wei and Pan, Feng and Yau, Shing-Tung},
	journal      = {The Journal of Physical Chemistry Letters},
	volume       = {14},
	number       = {4},
	pages        = {954--964},
	year         = {2023},
	doi          = {10.1021/acs.jpclett.2c03706},
	url          = {https://doi.org/10.1021/acs.jpclett.2c03706},
	note         = {Epub 2023-Jan-23}
}

@article{anand2022topological,
	title        = {Topological feature engineering for machine learning based halide perovskite materials design},
	author       = {Anand, D. Vijay and Xu, Qiang and Wee, Jun Jie and Xia, Kelin and Sum, Tze Chien},
	journal      = {npj Computational Materials},
	volume       = {8},
	number       = {1},
	pages        = {203},
	year         = {2022},
	doi          = {10.1038/s41524-022-00883-8},
	issn         = {2057-3960},
	url          = {https://www.nature.com/articles/s41524-022-00883-8}
}

@article{Hu2025CommAlgTDA,
	title   = {Commutative algebra-enhanced topological data analysis},
	author  = {Hu, Chuanshen and Wang, Yu and Xia, Kelin and Ye, Ke and Zhang, Yipeng},
	year    = {2025},
	month   = apr,
	journal = {arXiv preprint arXiv:2504.09174},
	eprint  = {2504.09174},
	doi     = {10.48550/arXiv.2504.09174},
	url     = {https://arxiv.org/abs/2504.09174},
	note    = {Submitted 2025-04-12; subjects: cs.CG, math.AC, math.AT}
}

@article{MillanArias2022DeLUCS,
	title   = {DeLUCS: Deep learning for unsupervised clustering of DNA sequences},
	author  = {Mill{\'a}n Arias, Pablo and Alipour, Fatemeh and Hill, Kathleen A. and Kari, Lila},
	journal = {PLOS ONE},
	volume  = {17},
	number  = {1},
	pages   = {e0261531},
	year    = {2022},
	month   = {01},
	publisher = {Public Library of Science},
	doi     = {10.1371/journal.pone.0261531},
	url     = {https://journals.plos.org/plosone/article?id=10.1371/journal.pone.0261531}
}

@article{SuwayyidWei2026_PSRT,
	author       = {Faisal Suwayyid and Guo-Wei Wei},
	title        = {Persistent {Stanley–Reisner} theory},
	journal      = {Foundations of Data Science},
	volume       = {8},
	pages        = {287--312},
	year         = {2026},
	doi          = {10.3934/fods.2025009},
	url          = {https://www.aimsciences.org/article/doi/10.3934/fods.2025009}
}

@book{bruns1998cohen,
	title        = {Cohen-Macaulay Rings},
	author       = {Winfried Bruns and H. Jürgen Herzog},
	series       = {Cambridge Studies in Advanced Mathematics},
	volume       = {39},
	edition      = {2nd},
	publisher    = {Cambridge University Press},
	address      = {Cambridge},
	year         = {1998},
	pages        = {468},
	isbn         = {978-0521566742},
	doi          = {10.1017/CBO9780511608681},
	url          = {https://doi.org/10.1017/CBO9780511608681}
}

@Book{MillerSturmfels2005,
	title        = {Combinatorial Commutative Algebra},
	author       = {Ezra Miller and Bernd Sturmfels},
	series       = {Graduate Texts in Mathematics},
	volume       = {227},
	publisher    = {Springer},
	address      = {New York},
	year         = {2005},
	isbn         = {978-0-387-22356-8},
	doi          = {10.1007/b138602},
	note         = {Hardcover edition: December 2004}  
}

@article{BubenikDlotko2017PersistenceLandscapes,
	author       = {Bubenik, Peter and Dłotko, Paweł},
	title        = {A persistence landscapes toolbox for topological statistics},
	journal      = {Journal of Symbolic Computation},
	volume       = {78},
	number       = {1},
	pages        = {91--114},
	year         = {2017},
	doi          = {10.1016/j.jsc.2016.03.009},
	url          = {https://doi.org/10.1016/j.jsc.2016.03.009}
}

@article{AdamsCoskunuzer2022_GeometricApproachesPersistentHomology,
	author       = {Adams, Henry and Coşkunüzer, Barış},
	title        = {Geometric Approaches to Persistent Homology},
	journal      = {SIAM Journal on Applied Algebra and Geometry},
	volume       = {6},
	number       = {4},
	pages        = {685--710},
	year         = {2022},
	doi          = {10.1137/21M1422914},
	url          = {https://epubs.siam.org/doi/10.1137/21M1422914}
}

@article{zia2025cap,
  title={CAP: Commutative algebra prediction of protein-nucleic acid binding affinities},
  author={Zia, Mushal and Suwayyid, Faisal Abdulaziz and Hozumi, Yuta and Wee, JunJie and Feng, Hongsong and Wei, Guo-Wei},
  journal={Machine Learning: Science and Technology},
  volume       = {6},
	pages        = {045068},
year         = {2025}
}

@article{CangWei2017_TopologyNet,
	author       = {Cang, Zixuan and Wei, Guo-Wei},
	title        = {TopologyNet: Topology based deep convolutional and multi-task neural networks for biomolecular property predictions},
	journal      = {PLoS Computational Biology},
	volume       = {13},
	number       = {7},
	pages        = {e1005690},
	year         = {2017},
	doi          = {10.1371/journal.pcbi.1005690},
	url          = {https://doi.org/10.1371/journal.pcbi.1005690}
}

@article{GrbicWuXiaWei2022_AspectsTopologicalApproaches,
	author       = {Grbić, Jelena and Wu, Jie and Xia, Kelin and Wei, Guo-Wei},
	title        = {Aspects of topological approaches for data science},
	journal      = {Foundations of Data Science},
	volume       = {4},
	number       = {2},
	pages        = {165--216},
	year         = {2022},
	doi          = {10.3934/fods.2022002},
	url          = {https://doi.org/10.3934/fods.2022002}
}

@article{chen2023persistent,
	title   = {Persistent hyperdigraph homology and persistent hyperdigraph {Laplacians}},
	author  = {Chen, Dong and Liu, Jian and Wu, Jie and Wei, Guo-Wei},
	journal = {Foundations of Data Science},
	year    = {2023},
	volume  = {5},
	number  = {4},
	pages   = {558--588},
	month   = dec,
	doi     = {10.3934/fods.2023010},
	url     = {https://doi.org/10.3934/fods.2023010}
}

@article{wee2023persistent,
	title     = {Persistent {Dirac} for molecular representation},
	author    = {Wee, Junjie and Bianconi, Ginestra and Xia, Kelin},
	journal   = {Scientific Reports},
	year      = {2023},
	volume    = {13},
	number    = {1},
	pages     = {11183},
	month     = jul,
	day       = {11},
	doi       = {10.1038/s41598-023-37853-z},
	url       = {https://doi.org/10.1038/s41598-023-37853-z},
	note      = {Article number: 11183}
}

@article{SuwayyidWei2024_PersistentMayerDirac,
	author       = {Suwayyid, Faisal and Wei, Guo-Wei},
	title        = {Persistent Mayer Dirac},
	journal      = {Journal of Physics: Complexity},
	volume       = {5},
	number       = {4},
	pages        = {045005},
	year         = {2024},
	doi          = {10.1088/2632-072X/ad83a5},
	url          = {https://doi.org/10.1088/2632-072X/ad83a5}
}

@misc{SuwayyidHozumiFengZiaWeeWei2025_CAKL,
	author       = {Suwayyid, Faisal and Hozumi, Yuta and Feng, Hongsong and Zia, Mushal and Wee, JunJie and Wei, Guo-Wei},
	title        = {CAKL: Commutative algebra k-mer learning of genomics},
	howpublished = {Nature Communication},
	year         = {2025},
	eprint       = {2508.09406},
	primaryClass = {q-bio.QM}
}

@article{FengSuwayyidZiaWeeHozumiChenWei2025_CAML,
	author       = {Feng, Hongsong and Suwayyid, Faisal and Zia, Mushal and Wee, Jun Jie and Hozumi, Yuta and Chen, Chia and Wei, Guo-Wei},
	title        = {CAML: Commutative Algebra Machine Learning---A Case Study on Protein-Ligand Binding Affinity Prediction},
	journal      = {Journal of Chemical Information \& Modeling},
	volume       = {65},
	number       = {13},
	pages        = {6732--6743},
	year         = {2025},
	doi          = {10.1021/acs.jcim.5c00940},
	url          = {https://doi.org/10.1021/acs.jcim.5c00940}
}

@misc{WeeSuwayyidZiaFengHozumiWei2025_CANet,
	author       = {Wee, JunJie and Suwayyid, Faisal and Zia, Mushal and Feng, Hongsong and Hozumi, Yuta and Wei, Guo-Wei},
	title        = {Commutative algebra neural network reveals genetic origins of diseases},
	howpublished = {arXiv preprint arXiv:2509.26566},
	year         = {2025},
	eprint       = {2509.26566},
	primaryClass = {q-bio.QM}
}

@article{Francisco2009_SplittingsOfMonomialIdeals,
	author       = {Francisco, Christopher A. and H\`a, H\^ong T. and Van Tuyl, Adam},
	title        = {Splittings of monomial ideals},
	journal      = {Proceedings of the American Mathematical Society},
	volume       = {137},
	number       = {10},
	pages        = {3271--3282},
	year         = {2009},
	doi          = {10.1090/S0002-9939-09-09929-8},
	url          = {https://doi.org/10.1090/S0002-9939-09-09929-8}
}

@incollection{VanTuyl2013_BeginnerGuideEdgeCoverIdeals,
	author       = {Van Tuyl, Adam},
	title        = {A Beginner’s Guide to Edge and Cover Ideals},
	booktitle    = {Monomial Ideals, Computations and Applications},
	editor       = {Bigatti, Anna M. and Gimenez, Philippe and Sáenz-de-Cabezón, Eduardo},
	series       = {Lecture Notes in Mathematics},
	volume       = {2083},
	pages        = {63--94},
	publisher    = {Springer Berlin Heidelberg},
	year         = {2013},
	doi          = {10.1007/978-3-642-38742-5\_3},
	url          = {https://doi.org/10.1007/978-3-642-38742-5\_3}
}

@book{Rotman2009_IntroHomologicalAlgebra,
	author       = {Rotman, Joseph J.},
	title        = {An Introduction to Homological Algebra},
	edition      = {2},
	series       = {Universitext},
	publisher    = {Springer-Verlag New York},
	address      = {New York, NY},
	year         = {2009},
	isbn         = {978-0-387-24527-0},
	doi          = {10.1007/b98977},
	url          = {https://doi.org/10.1007/b98977}
}

@book{Eisenbud1995_CommutativeAlgebra,
	author    = {Eisenbud, David},
	title     = {Commutative Algebra: With a View Toward Algebraic Geometry},
	series    = {Graduate Texts in Mathematics},
	volume    = {150},
	publisher = {Springer‐Verlag New York},
	address   = {New York, NY},
	year      = {1995},
	pages     = {xvi + 800},
	isbn      = {978‐0‐387‐94268‐1},
	doi       = {10.1007/978‐1‐4612‐5350‐1},
	url       = {https://doi.org/10.1007/978-1-4612-5350-1}
}

@book{HerzogHibi2011,
	author    = {Herzog, J{\"u}rgen and Hibi, Takayuki},
	title     = {Monomial Ideals},
	series    = {Graduate Texts in Mathematics},
	volume    = {260},
	publisher = {Springer},
	address   = {London},
	year      = {2011},
	doi       = {10.1007/978-0-85729-106-6},
	isbn      = {978-0-85729-105-9}
}

@misc{PubChem_CID643833,
	author       = {{National Center for Biotechnology Information}},
	title        = {PubChem Compound Summary for CID 643833: cis-1,2-Dichloroethylene},
	year         = {2025},
	howpublished = {PubChem},
	url          = {https://pubchem.ncbi.nlm.nih.gov/compound/cis-1_2-Dichloroethylene},
}

@misc{PubChem_CID638186,
	author       = {{National Center for Biotechnology Information}},
	title        = {PubChem Compound Summary for CID 638186: trans-1,2-Dichloroethylene},
	year         = {2025},
	howpublished = {PubChem},
	url          = {https://pubchem.ncbi.nlm.nih.gov/compound/Trans-1_2-Dichloroethylene},
}
\end{document}